\documentclass[a4paper,reqno,12pt]{amsart}
 
\usepackage{amsmath}
\usepackage{amssymb}
\usepackage{amsthm}
\usepackage[noadjust]{cite}
\usepackage{mathrsfs}
\usepackage[all,cmtip]{xy}
\usepackage{enumitem}
\usepackage{indentfirst}
\usepackage{geometry}
\usepackage{xcolor}
\allowdisplaybreaks[4]

\usepackage[colorlinks=true, citecolor=red, linkcolor=blue, filecolor=magenta, urlcolor=green,backref=page]{hyperref}

\newtheorem{thm}{Theorem}[section]
\newtheorem{prop}[thm]{Proposition}
\newtheorem{lem}[thm]{Lemma}

\newtheorem{cor}[thm]{Corollary}

\theoremstyle{definition}
\newtheorem{definition}[thm]{Definition}

\theoremstyle{remark}
\newtheorem{remark}[thm]{Remark}
\numberwithin{equation}{section}

\newcommand{\bC}{{\mathbb C}}
\newcommand{\bN}{\mathbb{N}}
\newcommand{\bP}{\mathbb{P}}
\newcommand{\bQ}{\mathbb{Q}}
\newcommand{\bR}{\mathbb{R}}
\newcommand{\bZ}{\mathbb{Z}}

\newcommand\cA{\mathcal{A}}

\newcommand\cC{\mathcal{C}}
\newcommand\cD{\mathcal{D}}
\newcommand\cE{\mathcal{E}}
\newcommand\cF{{\mathcal{F}}}
\newcommand\cG{{\mathcal{G}}}
\newcommand\cH{{\mathcal{H}}}
\newcommand\cI{{\mathcal{I}}}
\newcommand\cJ{{\mathcal{J}}}

\newcommand\cL{{\mathcal{L}}}

\newcommand\cN{{\mathcal{N}}}
\newcommand\cO{{\mathcal{O}}}
\newcommand\cP{\mathcal{P}}
\newcommand\cQ{{\mathcal{Q}}}
\newcommand\cR{{\mathcal{R}}}

\newcommand\cT{{\mathcal{T}}}

\newcommand\cV{{\mathcal{V}}}
\newcommand\cW{{\mathcal{W}}}
\newcommand\cX{{\mathcal{X}}}
\newcommand\cY{{\mathcal{Y}}}

\newcommand\fA{\mathfrak{A}}
\newcommand\fB{\mathfrak{B}}
\newcommand\fC{\mathfrak{C}}
\newcommand\fD{\mathfrak{D}}

\newcommand\KCF{K_{\mathcal{F}}}
\newcommand\KCG{K_{\mathcal{G}}}
\newcommand\KCH{K_{\mathcal{H}}}
\newcommand\KCJ{K_{\mathcal{J}}}
\newcommand\KFA{K_{\mathfrak{A}}}
\newcommand\KFB{K_{\mathfrak{B}}}
\newcommand\KFC{K_{\mathfrak{C}}}
\newcommand\KFD{K_{\mathfrak{D}}}

\newcommand\iB{B^{\inv}}
\newcommand\nB{B^{\ninv}}

\newcommand\nE{E^{\ninv}}
\newcommand\iL{L^{\inv}}
\newcommand\nL{L^{\ninv}}

\newcommand\Sing{\operatorname{Sing}}

\newcommand{\WDiv}{\operatorname{WDiv}}
\newcommand{\Supp}{\operatorname{Supp}}

\newcommand{\vol}{\operatorname{vol}}

\newcommand{\red}{\operatorname{red}}

\newcommand{\rank}{\operatorname{rank}}
\newcommand{\ninv}{\operatorname{ninv}}
\newcommand{\inv}{\operatorname{inv}}
\newcommand{\Int}{\operatorname{Int}}
\newcommand{\Exc}{\operatorname{Exc}}

\newcommand{\on}{\operatorname}

\title[Boundedness of Polarized Foliated Surfaces]{Boundedness of Polarized Foliated Surfaces}
\author{Yen-An Chen and Minzhe Zhu}
\address{School of Mathematics, Korea Institute for Advanced Study, Seoul, 02455, Korea}
\email{yachen@kias.re.kr}
\address{School of Mathematics, Korea Institute for Advanced Study, Seoul, 02455, Korea}
\email{zhumz@kias.re.kr}

\date{\today}

\begin{document}
\begin{abstract}
We establish the minimal model program for klt adjoint foliated surfaces and use it to study the boundedness of polarized adjoint foliated surfaces. We prove that $\epsilon$-lc adjoint foliated surfaces with nef adjoint canonical divisor and a nef and big integral polarization form a bounded family, provided that the volume of their sum is bounded from above. As applications, we establish boundedness and effective birationality for adjoint foliated surfaces of general type, as well as a uniform positive lower bound for their volumes. Finally, for $\epsilon$-lc Calabi--Yau adjoint foliated surfaces, we prove that the volumes of the canonical divisors of rank one foliations on the underlying surfaces belong to a fixed discrete set.
\end{abstract}
\maketitle

\tableofcontents

\section{Introduction}
Throughout this paper, we work over the field of complex numbers $\bC$. By \textit{integral divisor}, we mean a $\mathbb{Q}$-Cartier Weil divisor with integer coefficients that is not necessarily Cartier and need not be effective.
\vspace*{10pt}

Boundedness is a fundamental problem in birational geometry and an important first step toward the construction of moduli spaces. Landmark results include the boundedness of canonically polarized varieties \cite{HMX14ACCLCT,HMX18BoundednessModuliVarieties} and Birkar's proof of the Borisov–Alexeev–Borisov conjecture for Fano varieties with mild singularities \cite{Bir19AntiPluricanonicalSystemsFano,Bir21SingularitiesLinearSystems}. In settings where the canonical divisor does not itself provide sufficient positivity, it is natural to equip the variety with an additional polarization. This point of view has led to various boundedness results for polarized varieties and related structures; see, for example, \cite{Bir22ModuliAlgebraicVarieties,Bir23GeometryPolarisedVarieties,Jiang23BoundednessKltGoodMinimalModels,BDCS24BoundednessEllipticCalabiYau,Fil24BoundednessNfoldsKappa,FHS25BoundednessEllipticCalabi,HH25BoundednessModuliSpaces,jiao25BoundednessPolarizedCalabiYau,JJZ25BoundednessPolarizedLCYFibrations,Zhu25BoundednessStableMinimal}. 

Meanwhile, there has been much recent progress in the minimal model program for foliations; see \cite{CS21MMPCorankOneThreefolds,SS22LocalGlobalMMPCorankOneThreefolds,CHLX23MMPAlgIntFoliationGenPair,SS23EffectiveGenerationFoliatedSurfaces,Cas+24MMPAlgIntAdjFolStr,Cas+25OnFiniteGenerationAndBoundedness,CS25MMPRankOneThreefolds,CS25MMPAlgIntFoliations,LMX25MMPAlgebraicallyIntegrable}. This makes it natural to investigate analogous boundedness questions in the foliated setting. However, the birational geometry governed by the canonical divisor $\KCF$ of a foliation $\cF$ differs substantially from the classical theory. Many of the techniques that are indispensable in the standard boundedness theory fail when one considers the canonical divisor of a foliation alone. For instance, vanishing theorem of Kawamata–Viehweg-type is generally unavailable, the abundance conjecture is known to fail in general (\cite[Example 5.5]{ACSS21PositivityModuliPart} and \cite[Theorem IV.2.2]{McQ08CanonicalModelsFoliations}), and the bigness of $K_\cF$ does not necessarily imply effective birationality (\cite[Theorem 1.3]{Lu25Unboundedness}). 
Furthermore, the Fano case also behaves differently: there is no klt Fano foliation of rank one on a normal projective surface (\cite[Proposition 5.3]{AD13FanoFoliations}), whereas lc Fano foliations of rank one on normal projective surfaces are unbounded (\cite[Example 3.11]{CJV24ExistenceComplementsFoliations}). 
These difficulties indicate that the canonical divisor of the foliation alone does not provide a sufficiently flexible framework for extending the classical boundedness theory. 

Due to these obstacles, it is natural to enlarge the category of objects under consideration by introducing adjoint foliated structures $\fA:=(X,\cF,t)$, where $X$ is a normal projective variety, $\cF$ is a foliation on $X$, and $t\in [0,1]$ is a real number; see \cite{PS19EffectiveAlgebraicIntegration,SS23EffectiveGenerationFoliatedSurfaces,Cas+24MMPAlgIntAdjFolStr,Cas+25OnFiniteGenerationAndBoundedness,Vas25ExplicitBoundFoliatedSurface,Xu25NumericalReductionSharpThresholds,CLSS26BirationalBoundednessOfStableFamilies}. 
The corresponding adjoint canonical divisor $\KFA:=tK_\cF+(1-t)K_X$ simultaneously records the geometry of the foliation and that of the ambient variety. Consequently, adjoint foliated structures provide a unified setting in which techniques from both the classical and foliated minimal model programs can be combined. In this paper, we establish the minimal model program for adjoint foliated surfaces and study their boundedness after equipping the underlying surface with a polarization.

\subsection{MMP for adjoint foliated surfaces}
The MMP for adjoint foliated surfaces has previously been established under additional assumptions. Spicer and Svaldi proved an MMP for adjoint divisors of the form $\KCF+\lambda K_X$, under an adjoint log canonical assumption and with $\lambda>0$ sufficiently small; see \cite[Theorem~3.1]{SS23EffectiveGenerationFoliatedSurfaces}. More recently, Vassiliadis established an adjoint MMP assuming that the foliation is log canonical or non-dicritical; see \cite[Theorem~1.5]{Vas25ExplicitBoundFoliatedSurface}. We also refer to \cite{Cas+24MMPAlgIntAdjFolStr} for the MMP for algebraically integrable adjoint foliated structures of arbitrary rank and dimension.

Our first main result extends the surface results of Spicer--Svaldi and Vassiliadis to klt adjoint foliated surfaces. In particular, we do not require the foliation itself to be log canonical or non-dicritical, nor do we require the coefficient of $K_X$ to be sufficiently small. This MMP provides the birational framework needed to prove the boundedness results later in the paper.

\begin{thm}[{$=$ Theorem~\ref{adjoint MMP}}]\label{mainthm: adjoint MMP}
    Let $\fA:=(X,\cF,t)$ be a klt adjoint foliated surface. Then we may run a $\KFA$-MMP $f:X\to Y$. Moreover, the following properties hold:
    \begin{enumerate}
        \item $\fB:=(Y,f_*\cF,t)$ is klt.
        \item If $\KFA$ is pseudo-effective, then $\KFB$ is nef.
        \item If $\KFA$ is not pseudo-effective, then $Y$ admits a fibration $Y\to Z$ such that $\rho(Y/Z)=1$ and $-\KFB$ is ample over $Z$.
    \end{enumerate}
\end{thm}

\subsection{Boundedness results}
We now turn to our main boundedness theorem for polarized adjoint foliated surfaces. See Definition~\ref{def of boundedness} for the definition of boundedness for adjoint foliated surfaces.
\begin{thm}[{$=$ Theorem~\ref{bdd of surfaces}}]\label{mainthm: bdd of surfaces without pairs}
Let $\epsilon, v$ be positive real numbers and $I\subset \bQ\cap(0,1)$ be a finite set. 
We consider adjoint foliated surfaces $\fA:=(X,\cF,t)$ and integral divisors $N$ on $X$ such that
    \begin{enumerate}
        \item $\fA$ is $\epsilon$-lc,
        \item $t\in I$,
        \item $\KFA$ is nef,
        \item $N$ is nef and big, and
        \item $\vol(\KFA+N)\leq v$.
    \end{enumerate}
    Then the set of such adjoint foliated surfaces $(X,\cF,t)$ forms a bounded family. If in addition $N\geq 0$, then the set of such adjoint foliated surfaces $(X,\cF,\Supp N,t)$ forms a log bounded family. 
\end{thm}

Spicer and Svaldi proved a boundedness result for canonical adjoint foliated surfaces, assuming that the adjoint coefficient is fixed and sufficiently small and that the adjoint canonical divisor is ample; see \cite[Theorem~1.3]{SS23EffectiveGenerationFoliatedSurfaces}.

As a first application, we obtain the following corollary which extends their result to $\epsilon$-lc adjoint foliated surfaces, allows the adjoint coefficient to vary in a fixed finite subset, and only requires the adjoint canonical divisor to be nef and big.

\begin{cor}[{$=$ Corollary~\ref{bdd of surfaces of general type}}]\label{main thm:bdd of surfaces of general type}
    Let $\epsilon, v$ be positive real numbers and $I\subset\bQ\cap(0,1)$ be a finite set. 
    We consider adjoint foliated surfaces $\fA:=(X,\cF,t)$ such that
    \begin{enumerate}
        \item $\fA$ is $\epsilon$-lc,
        \item $t\in I$,
        \item $\KFA$ is nef and big, and
        \item $\vol(\KFA)\leq v$.
    \end{enumerate}
    Then the set of such adjoint foliated surfaces $(X,\cF,t)$ forms a bounded family.
\end{cor}

A key ingredient in our boundedness arguments is the following effective birationality theorem for adjoint foliated surfaces of general type.

\begin{thm}[{$=$ Corollary~\ref{effective birationality of general type}}]\label{main thm:effective birationality of general type}
    Let $\epsilon$ be a positive real number and $I\subset \bQ\cap(0,1)$ be a finite set. Then there exists a positive integer $m$ depending only on $\epsilon,I$ satisfying the following: 

    If $\fA:=(X,\cF,t)$ is an $\epsilon$-lc adjoint foliated surface with $t\in I$ and $\KFA$ is big, then $|m\KFA|$ defines a birational map.
\end{thm}

Uniform lower bounds for canonical volumes form an important part of the classical boundedness theory; see, for example, \cite[Theorem~1.3]{HMX14ACCLCT}. In the foliated setting, Spicer and Svaldi proved a uniform lower bound for the volumes of adjoint divisors associated with canonical foliations and a fixed sufficiently small adjoint parameter; see \cite[Theorem~1.5]{SS23EffectiveGenerationFoliatedSurfaces}. More recently, an analogous lower bound was established for lc algebraically integrable adjoint foliated structures in arbitrary dimension; see \cite[Theorem~2.3]{CLSS26BirationalBoundednessOfStableFamilies}. 

As another consequence of the effective birationality theorem above, we obtain the following lower bound for $\epsilon$-lc adjoint foliated surfaces.

\begin{cor}[{$=$ Corollary~\ref{lower bound of volume of general type}}]\label{mainthm:lower bound of volume of general type}
Let $\epsilon$ be a positive real number, and $I\subset \bQ\cap(0,1)$ be a finite set. Then there exists a positive real number $v$ depending only on $\epsilon, I$ satisfying the following: 

If $\fA:=(X,\cF,t)$ is an $\epsilon$-lc adjoint foliated surface with $t\in I$ and $\KFA$ is big, then $\vol(\KFA)\geq v$. 
\end{cor}

In a different direction, Jiao proved that the volumes of integral divisors on $\epsilon$-lc Calabi--Yau pairs belong to a fixed discrete set; see \cite[Theorem~1.1]{Jiao25Discretenessvolume}.

In the same spirit as Jiao’s result, we prove the following discreteness statement for the canonical divisors of rank one foliations on surfaces admitting a Calabi--Yau adjoint foliated structure.

\begin{cor}[{$=$ Corollary~\ref{discreteness of volume of CY}}]\label{mainthm:discreteness of volume of CY}
Let $\epsilon$ be a positive real number and $I\subset \bQ\cap(0,1)$ be a finite set. Then there exists a discrete set $J\subset \bQ^{>0}$ depending only on $\epsilon, I$ satisfying the following: 

Assume that
\begin{enumerate}
\item $\fA:=(X,\cF,t)$ is an $\epsilon$-lc adjoint foliated surface with $t\in I$,
\item $\KFA\sim_{\bQ}0$, and
\item $\cG$ is a rank one foliation on $X$.
\end{enumerate}
Then $\vol(\KCG)\in J\cup\{0\}$. In particular, if $\KCG$ is big, then $\vol(\KCG)$ is bounded from below by $\min J$.
\end{cor}

\subsection{Sketch of proof}
We first explain the proof of the MMP for klt adjoint foliated surfaces. For simplicity of notation, we only consider the case without boundaries. Let $\fA:=(X,\cF,t)$ be a klt adjoint foliated surface and $\KFA=t\KCF+(1-t)K_X$. Let $R\subset \overline{\textnormal{NE}}(X)$ be a $\KFA$-negative extremal ray. Then the ray $R$ is either $K_X$-negative or $\KCF$-negative. Note that by Proposition~\ref{control singularities}, the underlying space $X$ is klt. If $R$ is $K_X$-negative, then it can be contracted by the classical contraction theorem for surfaces. Therefore, the main point is to treat the case where $R$ is $\KCF$-negative.

The difficulty in contracting a $K_\cF$-negative extremal ray in our setting is that the foliation $\cF$ itself may not be log canonical, so one cannot directly apply the usual contraction theorem for foliated surfaces. We treat divisorial and fiber type contractions separately.

First we consider the divisorial case. We pass to an F-dlt modification $\pi:W\to X$ of $\cF$, which is constructed by Cascini and Spicer; see \cite[Theorem 8.1]{CS21MMPCorankOneThreefolds}. If $C$ is a curve spanning the extremal ray $R$, then its strict transform $C_W$ intersects the induced foliated pair negatively. Hence we can contract $C_W$ on $W$ by the contraction theorem for lc foliated pairs. After this contraction, there may still be some exceptional divisors over $X$ left. We then run suitable auxiliary MMPs to contract these remaining exceptional divisors. Then the resulting birational map descends to a birational morphism from $X$. This gives the desired divisorial contraction of $R$.

The fiber type case is treated by a different method. Using the foliated bend-and-break method \cite[Corollary 2.28]{Spi20HigherDimensionalFoliatedMori}, one can show that there is a moving family of rational curves whose numerical class spans $R$. Note that \cite[Corollary 2.28]{Spi20HigherDimensionalFoliatedMori} does not require $\cF$ to be lc. This allows us to prove that $K_X\cdot R<0$. Therefore, the desired contraction follows from the classical contraction theorem applied to the underlying surface. This produces the required fiber type contraction.

Combining the divisorial and fiber type cases, every $\KFA$-negative extremal ray can be contracted. Since each birational contraction decreases the Picard number by one, the MMP terminates after finitely many steps. If $\KFA$ is pseudo-effective, the process ends with a model on which the adjoint canonical divisor is nef; otherwise, it ends with a Mori fiber space.

We now turn to the proof of the boundedness result for polarized adjoint foliated surfaces with bounded volumes. One of the main tools is Theorem~\ref{Deriving boundedness from birational boundedness}, which derives boundedness of adjoint foliated surfaces of general type from birational boundedness.

Suppose that our adjoint foliated surfaces are already log birationally bounded. We choose a bounded birational model appearing as a fiber of a fixed family. In the classical setting of pairs, one usually takes a simultaneous log resolution, runs a relative MMP, and then uses invariance of plurigenera to control the resulting canonical models fiberwise; see, for example, the proof of \cite[Theorem 1.6]{HMX14ACCLCT}. In the foliated setting, however, each of these three steps presents substantial difficulties.

First, even after stratifying the base, one cannot expect a simultaneous resolution on which the induced foliated pair on every fiber is foliated log smooth, or even has canonical singularities. Indeed, non-canonical foliation singularities may occur over a Zariski dense subset of the base; see \cite[\S 7.2]{PS19EffectiveAlgebraicIntegration}. To overcome this obstruction, we use a result of Pereira and Svaldi which shows that, after shrinking the base, there exists a simultaneous resolution on which the induced foliations on all fibers become uniformly $\delta$-canonical, and hence log canonical; see Definition~\ref{delta-canonical} for the notion of $\delta$-canonicity. Combined with the weaker notion of foliation-adapted log smoothness introduced in Definition~\ref{foliation-adapted log smooth}, this gives sufficient uniform control of the singularities in the family.

Second, the MMP developed here for adjoint foliated structures is a surface result, so it cannot be run on the higher-dimensional total space of the family. Our strategy is therefore to run it on the generic fiber. Before doing so, we need a big adjoint canonical divisor on the generic fiber corresponding to a fixed adjoint parameter. The adjoint coefficient $t$ associated with the original surface may vary among fibers. Hence, bigness on a closed fiber does not automatically imply bigness on the generic fiber. By \cite[Lemma~5.3]{LX25NonalgebraictyNonabundant}, the foliated part of the adjoint canonical divisor is pseudo-effective. We may therefore replace the varying parameter by the fixed value $1-\delta$ while preserving bigness. Effective birationality (Corollary~\ref{effective birationality of general type}) and semicontinuity then show that the corresponding fixed adjoint canonical divisor is big on the generic fiber. This provides a fixed big point in the parameter polytope.

Third, we need to establish bigness on the generic fiber uniformly over the entire parameter polytope. More precisely, whenever a parameter gives a big adjoint canonical divisor on one of the relevant closed fibers, we need the corresponding divisor on the generic fiber to be big as well. In the classical setting, one would usually use invariance of plurigenera \cite[Theorem 1.8]{HMX13OnBirationalAutomorphisms} to pass from the closed fiber to the generic fiber. However, no suitable invariance-of-plurigenera theorem is available for adjoint foliated structures, while semicontinuity only yields an open subset depending on the chosen parameter and therefore does not give a single open subset of the base that works uniformly for all parameters.

The key ingredient is Lemma~\ref{bigness on the generic fiber}, which provides a substitute using the finiteness of good minimal models and ample models established in Theorem~\ref{finiteness of models}, together with specialization of movable curves. Indeed, suppose that the desired statement fails for some parameter. Join this parameter to the fixed big point constructed above, and consider the last point along this segment at which the divisor on the generic fiber remains pseudo-effective. Then the divisor at this point is pseudo-effective but not big. One of the finitely many ample model diagrams then yields a covering family of curves on which the induced divisor has intersection number zero. Comparing intersections along the segment shows that the divisor corresponding to the original parameter has non-positive intersection with the same family of curves. After spreading out and specializing this family to the relevant closed fiber, this contradicts the bigness of the divisor there. Thus every parameter which gives a big divisor on one of the relevant closed fibers also gives a big divisor on the generic fiber.

A further uniformity issue arises when we extend the ample models on the generic fiber back to the family. Both the boundary and the adjoint coefficient vary in a polytope. Consequently, applying openness of ampleness separately to every induced divisor would require treating infinitely many divisors and would not yield a single open subset of the base on which all the models work simultaneously.

We resolve this issue by a finite polyhedral argument. Theorem~\ref{finiteness of models} and Corollary~\ref{finite model decomposition} give finitely many ample model diagrams on the generic fiber, organized into a finite polyhedral decomposition compatible along faces. After replacing the base by a finite cover and shrinking it, we spread out these diagrams. For each polytope, we choose one point in its relative interior and apply openness of ampleness only to the divisor corresponding to this fixed point. Now let any other point in the relative interior be given. It can be written as a positive convex combination of the chosen interior point and a point lying in the relative interior of a proper face. By induction on the dimension and compatibility along faces, the divisor corresponding to the latter point is the pullback of an ample divisor on the model associated with that face, and is therefore nef on the larger model. The divisor corresponding to the given interior point is thus a positive convex combination of an ample divisor and a nef divisor, and hence is ample. It follows that the spread-out diagrams give the ample models on every fiber for all parameters simultaneously. This avoids having to spread out infinitely many ample divisors individually. Finally, by the uniqueness of ample models, the original adjoint foliated surfaces are recovered as fibers of finitely many bounded families.

We then return to the proof of the boundedness theorem. Let $\fA:=(X,\cF,B,t)$ be an adjoint foliated surface appearing in Theorem~\ref{bdd of surfaces}, and let $N$ be the polarization. Since $\KFA$ is nef and $N$ is nef and big, the divisor $\KFA+N$ is big. By Birkar's effective birationality result for adjoint linear series \cite[Corollary 1.2]{Bir23GeometryPolarisedVarieties}, there exist fixed positive integers $m$ and $n$ such that
$|m\KFA+nN|$ defines a birational map. Together with the upper bound on $\vol(\KFA+N)$, this gives log birational boundedness of the adjoint foliated surfaces under consideration.

It remains to upgrade this log birational boundedness to log boundedness by Theorem~\ref{Deriving boundedness from birational boundedness}. For this purpose, we need to incorporate a general element $L\in |m\KFA+nN|$ into the boundary to obtain an adjoint foliated surface of general type. The key point is to prove a uniform positive lower bound for the log canonical threshold of $L$ with respect to $\fA$; see Theorem~\ref{lower bound of lct}. More precisely, there exists a fixed constant $\tau>0$ such
that the adjoint foliated structure $\fA'=(X,\cF,B+\tau(\nL+\frac{1}{1-t}\iL),t)$ still has controlled singularities. Moreover, $K_{\fA'}=\KFA+\tau L$ is big. Thus $\fA'$ is an adjoint foliated surface of general type. Therefore, the boundedness of $\fA$ follows from Theorem~\ref{Deriving boundedness from birational boundedness}.

\vspace{0.5cm}
{\textbf{\sffamily{Structure of the paper.}}} This paper is organized as follows. In \S\ref{Section:Preliminaries}, we collect preliminary results on foliations, adjoint foliated surfaces, types of birational models, and boundedness. In \S\ref{Section:MMP}, we establish the MMP for klt adjoint foliated surfaces and prove the existence of good minimal models and ample models. In \S\ref{Section:Finiteness}, we prove the finiteness of good minimal models and ample models as the boundary and the adjoint coefficient vary in a polytope, together with a uniform generic-fiber bigness result. Finally, in \S\ref{Section:Boundedness}, we establish effective birationality, derive boundedness from birational boundedness, and prove our main boundedness theorem and its applications to volumes.

\vspace{0.5cm}
{\textbf{\sffamily{Acknowledgement.}}} 
We would like to thank Zhengyang Cui, Calum Spicer, Roberto Svaldi, and Lingyao Xie for many valuable conversations, suggestions, and comments. We are particularly grateful to Calum Spicer and Roberto Svaldi for answering a question concerning \cite{SS23EffectiveGenerationFoliatedSurfaces}. We are also grateful to Lingyao Xie for answering a question concerning \cite{Cas+25OnFiniteGenerationAndBoundedness}. The first and second authors are supported by the KIAS Individual Grants (MG102901 and MG106901, respectively) at Korea Institute for Advanced Study.

\section{Preliminaries}\label{Section:Preliminaries}
\subsection{Coefficients}
If $I\subseteq\bR$ and $B$ is an $\bR$-divisor, we write $B\in I$ if every non-zero coefficient of $B$ belongs to $I$.
\subsection{Foliations}
Let $X$ be a normal variety. 
A \emph{foliation} $\cF$ is a coherent subsheaf $T_\cF$ of the tangent sheaf $T_X$ such that
\begin{enumerate}
\item $T_\cF$ is saturated, i.e., $T_X/T_\cF$ is torsion-free, and
\item $T_\cF$ is closed under the Lie bracket.
\end{enumerate}

Let $r=\on{rank}(T_\cF)$ be the \emph{rank} of the foliation. 
The \emph{canonical divisor} $K_\cF$ is a Weil divisor on $X$ such that $\cO_X(-K_\cF)\cong\det T_\cF$. 

Let $\pi: Y \dashrightarrow X$ be a dominant rational map between normal varieties and $\cF$ be a foliation on $X$. 
We denote by $\pi^{-1}\cF$ the \emph{pullback foliation} on $Y$ (see, for example, \cite[Section 3.2]{Dru21Codim1FolNumTriCanClass}). 
If $f: X\dashrightarrow X'$ is birational, then $f_*\mathcal{F}$ represents the transformed foliation on $X'$ induced by $f^{-1}$. 

Let $X^\circ$ be the open subset of $X$ such that $\cF\vert_{X^\circ}$ is a subbundle of $T_{X^\circ}$. 
A \emph{leaf} $L$ is a maximal connected and immersed holomorphic submanifold $L\subseteq X^\circ$ such that $T_L = \cF\vert_L$. 

A foliation $\cF$ is called \emph{algebraically integrable} if all of its leaves are algebraic. 
Equivalently, an algebraically integrable foliation $\cF$ on $X$ is induced by a dominant rational map $f: X\dashrightarrow Y$ for some normal variety $Y$ (see, for example, \cite[Sections 3.2 and 3.6]{Dru21Codim1FolNumTriCanClass}).

\begin{definition}[Singular locus]
Let $\cF$ be a foliation of rank $r$ on a normal variety $X$. 
We obtain a morphism $\varphi: (\Omega_X^r)^{**} \to \cO_X(K_\cF)$ by taking the double dual of the $r$-th wedge product of $\Omega_X^{**}\to \cF^*$, which is induced by the inclusion $\cF\subseteq \cT_X$. 
We define the \emph{singular locus} of $\cF$, denoted by $\Sing(\cF)$, as the cosupport of the image of $\varphi': (\Omega_X^r\otimes\cO_X(-K_\cF))^{**} \to \cO_X$. 
\end{definition}

\begin{definition}[Invariance]\label{defn_invariance}
Let $\cF$ be a foliation on a normal variety $X$. 
\begin{enumerate}
    \item We say that a subvariety $S\subseteq X$ is $\cF$-\emph{invariant} if for any open subset $U\subseteq X$ and any section $\partial\in H^0(U,T_\cF)$, we have 
    $\partial(\cI_{S\cap U})\subseteq\cI_{S\cap U}$, where $\cI_{S\cap U}$ is the ideal sheaf of $S\cap U$. 
    \item For any birational morphism $\pi:Y\to X$ and for any prime divisor $D$ on $Y$, we define $\epsilon_\cF(D)=0$ if $D$ is $\pi^{-1}\cF$-invariant, and $\epsilon_\cF(D)=1$ if $D$ is not $\pi^{-1}\cF$-invariant. 
    \item For any $\bR$-divisor $D=\sum a_iD_i$ on $X$, where $D_i$ are prime divisors, we define $D^{\ninv}:=\sum\epsilon_\cF(D_i)a_iD_i$, and $D^{\inv}:=D-D^{\ninv}$. 
\end{enumerate}
\end{definition}

\begin{definition}[Non-dicritical]
    Let $\cF$ be a foliation of rank one on a normal surface $X$. 
    We say $\cF$ is \emph{non-dicritical} if $\epsilon_\cF(E)=0$ for all exceptional divisors $E$ over $X$. 
\end{definition}

\subsection{Adjoint foliated surfaces}
The following definition follows from {\cite{Cas+25OnFiniteGenerationAndBoundedness}}: 
\begin{definition}\label{defn:adjoint_fol_str}
    We say $\fA:=(X,\cF,B,t)$ is a \emph{sub-adjoint foliated structure} if the following conditions hold:
\begin{enumerate}
    \item $X$ is a normal projective variety, 
    \item $\cF$ is a foliation on $X$, 
    \item $B$ is an $\bR$-divisor, 
    \item $t\in [0,1]$, and
    \item the adjoint log canonical divisor $K_\fA:= t(K_\cF+\nB) + (1-t)(K_X+B)$ is $\bR$-Cartier.
\end{enumerate}
When $B=0$, we write $(X,\cF,t)$ instead of $(X,\cF,0,t)$.

If $B$ is effective, we say $\fA$ is an \emph{adjoint foliated structure}. 
In particular, $\fA=(X,\cF,B,t)$ is called an \emph{adjoint foliated surface} if $\dim X=2$ and $\rank\cF=1$. 
\end{definition}

\begin{remark}
    Note that Definition~\ref{defn:adjoint_fol_str} is equivalent to the notion \emph{normalized adjoint foliated structure} in \cite[Definition 3.12]{CLSS26BirationalBoundednessOfStableFamilies}. 
\end{remark}

\begin{definition}
Let $\fA=(X,\cF,B,t)$ be an adjoint foliated structure. Let $E$ be a prime divisor over $X$. Let $\pi:X'\to X$ be a birational morphism such that $E$ is a prime divisor on $X'$, and denote $\cF':=\pi^{-1}\cF$. We define the \emph{discrepancy} of $E$ with respect to $\fA$ to be \[a(E,\fA):=\on{mult}_E (tK_{\cF'}+(1-t)K_{X'}-\pi^*K_\fA),\]
and define the \emph{log discrepancy} of $E$ with respect to $\fA$ to be 
\[A(E,\fA):=a(E,\fA)+t\epsilon_\cF(E)+(1-t).\] 

We say $\fA$ is \emph{canonical} (resp. \emph{terminal}) if $a(E,\fA)\geq 0$ (resp. $>0$) for any exceptional divisor $E$ over $X$. We say $\fA$ is \emph{lc} (resp. \emph{klt}, \emph{$\epsilon$-lc}) if $A(E,\fA)\geq 0$ (resp. $>0$, $\geq\epsilon$) for any prime divisor $E$ over $X$. 

If $t=0$, we write $\fA$, $a(E,\fA)$, $A(E,\fA)$ as $(X,B)$, $a(E,X,B)$, $A(E,X,B)$, respectively. 

If $t=1$, we write $\fA$, $a(E,\fA)$, $A(E,\fA)$ as $(\cF,\nB)$, $a(E,\cF,\nB)$, $A(E,\cF,\nB)$, respectively. 
\end{definition}

\begin{prop}\label{control singularities}
Let $\fA:=(X,\cF,B,t)$ be a klt adjoint foliated surface. Then $(X,B)$ is klt and
\[
A(E,X,B)\geq A(E,\cF,\nB)
\]
for any prime divisor $E$ over $X$. In particular, if $\fA$ is $\epsilon$-lc for some positive real number $\epsilon$, then $(X,B)$ is $\epsilon$-lc.

\begin{proof}
For any prime divisor $E$ over $X$, we have
\[A(E,\fA)=tA(E,\cF,\nB)+(1-t)A(E,X,B).\]
We prove that $A(E,X,B)\geq A(E,\cF,\nB)$. It then follows that \[A(E,X,B)\geq A(E,\fA)>0,\] and hence $(X,B)$ is klt. Moreover, if $\fA$ is $\epsilon$-lc, then
\[A(E,X,B)\geq A(E,\fA)\geq\epsilon,\]
so $(X,B)$ is $\epsilon$-lc.

\begin{enumerate}[label=\textsl{Step} \arabic*., wide=13pt, itemsep=13pt]

\item In this step, we first consider the case $t=1$.

Let $D$ be a prime divisor on $X$. Since $\fA=(\cF,\nB)$ is klt, we have
\[0\leq-a(D,\cF,\nB)<\epsilon_\cF(D).\]
Thus $\epsilon_\cF(D)=1$. Hence no prime divisor on $X$ is $\cF$-invariant, and therefore $B=\nB$. It follows that
\[A(D,X,B)=A(D,\cF,\nB).\]

For any exceptional divisor $E$ over $X$, we have
\[a(E,\cF)\geq a(E,\cF,B)>-\epsilon_\cF(E),\]
so $\cF$ is log canonical. We claim that $\cF$ is non-dicritical. Otherwise, the minimal resolution of $\cF$ contains a non-$\cF$-invariant exceptional divisor $E_0$. By the classification of log canonical foliation singularities on surfaces \cite[Theorem~1.1]{Che23LogCanonicalFoliationSurfaces}, $E_0$ is a strictly log canonical place, that is, $a(E_0,\cF)=-1$. This contradicts $a(E_0,\cF)>-1$. Therefore, $\cF$ is non-dicritical.

Consequently, $\epsilon_\cF(E)=0$ for every exceptional divisor $E$ over $X$, and hence
\[A(E,\cF,\nB)=a(E,\cF,B).\]
By \cite[Proposition~5.1(1)]{CS26RecentProgress}, we obtain
\[A(E,X,B)\geq a(E,\cF,B)=A(E,\cF,\nB).\]
This proves the desired inequality when $t=1$.

\item In this step, we consider the case $t<1$.

Let $D$ be a prime divisor on $X$. If $D$ is not $\cF$-invariant, then $B$ and $\nB$ have the same coefficient along $D$, and hence
\[A(D,X,B)=A(D,\cF,\nB).\]
If $D$ is $\cF$-invariant, then $A(D,\cF,\nB)=0$. On the other hand,
\[0<A(D,\fA)=(1-t)A(D,X,B),\]
so $A(D,X,B)>0=A(D,\cF,\nB)$. Thus the desired inequality holds for every prime divisor on $X$.

It remains to consider exceptional divisors over $X$. By \cite[Theorem~8.1]{CS21MMPCorankOneThreefolds}, $(\cF,\nB)$ admits an F-dlt modification $\pi\colon W\to X$ such that $W$ is $\bQ$-factorial and the induced foliation $\cG:=\pi^{-1}\cF$ is non-dicritical. Let $B_W$ be the strict transform of $B$ on $W$, and let $E$ be the sum of reduced $\pi$-exceptional divisors. For any exceptional divisor $P$ over $X$, the definition of an F-dlt modification gives
\[a(P,\cF,\nB)\leq a(P,\cG,\nB_W+\nE).\]
We divide the proof into two cases:

\begin{enumerate} [label=\textsl{Case} \arabic*., wide=13pt, itemsep=13pt]
\item  Assume that $P$ is a divisor on $W$. Since $(\cG,\nB_W+\nE)$ is F-dlt, we have
\[a(P,\cG,\nB_W+\nE)=-\epsilon_\cF(P).\]
Therefore,
\[a(P,\cF,\nB)\leq-\epsilon_\cF(P),\]
and hence $A(P,\cF,\nB)\leq0$. Since $\fA$ is klt, we have
\[0<A(P,\fA)=tA(P,\cF,\nB)+(1-t)A(P,X,B).\]
Combining the inequalities above, we obtain \[A(P,X,B)>0\geq A(P,\cF,\nB).\]
In particular, we may write
\[K_W+B_W+E=\pi^*(K_X+B)+G,\]
where $G$ is an effective $\pi$-exceptional divisor.

\item Assume that $P$ is exceptional over $W$. Since $\cG$ is non-dicritical, we have $\epsilon_\cF(P)=\epsilon_\cG(P)=0$. By \cite[Proposition~5.1(1)]{CS26RecentProgress}, we obtain
\[
\begin{aligned}
A(P,\cF,\nB)
&=a(P,\cF,\nB)\\
&\leq a(P,\cG,\nB_W+\nE)\\
&\leq A(P,W,B_W+E)\\
&\leq A(P,X,B).
\end{aligned}
\]
Here the last inequality follows from the effectivity of $G$.

Thus $A(E,X,B)\geq A(E,\cF,\nB)$ for every prime divisor $E$ over $X$. As explained at the beginning of the proof, it follows that $(X,B)$ is klt. If $\fA$ is $\epsilon$-lc, then $(X,B)$ is also $\epsilon$-lc.
\end{enumerate}
\end{enumerate}
\end{proof}
\end{prop}

\begin{definition}[{\cite[Definitions~4.1 and~4.2]{PS19EffectiveAlgebraicIntegration}}]\label{delta-canonical}
Let $\delta$ be a positive real number. Let $\cF$ be a foliation on a normal projective surface $X$. We say that $\cF$ is \emph{$\delta$-canonical} if $(X,\cF,t)$ is a canonical adjoint foliated surface for any $0\leq t\leq 1-\delta$.
\end{definition}

\begin{remark}\label{delta-canonical implies lc}
By \cite[Proposition~4.9]{PS19EffectiveAlgebraicIntegration}, any $\delta$-canonical foliation on a smooth projective surface with $0<\delta<\frac14$ is log canonical.
\end{remark}

Although one cannot in general construct a simultaneous resolution of a family such that the induced foliation on every fiber has canonical singularities, the following lemma shows that, after shrinking the base, one can modify the family uniformly so that the induced foliation on every fiber is $\delta$-canonical.

\begin{lem}[{\cite[Proposition~7.4]{PS19EffectiveAlgebraicIntegration}}]\label{delta-canonical family}
Let $\pi:\cX\to Z$ be a family of smooth projective surfaces, and let $\cF$ be a rank one foliation on $\cX$ tangent to the fibers of $\pi$. Fix a real number $0<\delta<\frac14$. Then, after replacing $Z$ by a nonempty open subset, there exists a finite sequence of blow-ups along multi-sections $f:\cY\to\cX$ such that, setting $\cG:=f^{-1}\cF$, the foliation $\cG_z$ is $\delta$-canonical, and hence log canonical, for every closed point $z\in Z$.
\end{lem}

\begin{definition}[Foliation-adapted log smooth]\label{foliation-adapted log smooth}
Let $X$ be a normal projective surface, $\cF$ be a rank one foliation, and $B$ be an effective $\bR$-divisor on $X$. We say that $(X,B)$ is \emph{foliation-adapted log smooth} (or \emph{$\cF$-adapted log smooth}) if
\begin{enumerate}
    \item $(X,B)$ is log smooth,
    \item every irreducible component of $\Supp(\nB)$ is everywhere transverse to $\cF$, and
    \item the irreducible components of $\Supp(\nB)$ are pairwise disjoint.
\end{enumerate}
\end{definition}

\begin{remark}
The reason for introducing this notion is that the usual notion of foliated log smoothness, as in \cite[Definition 3.1]{CS21MMPCorankOneThreefolds}, is too strong in families. As illustrated in \cite[\S7.2]{PS19EffectiveAlgebraicIntegration}, even after a stratification of the base, one cannot expect a simultaneous resolution whose fibers are all foliated log smooth, since non-canonical foliation singularities may occur along a Zariski dense subset of the base. Our notion of foliation-adapted log smoothness avoids this issue: it only controls the relative position of the boundary and the foliation. These conditions are open and, after passing to a stratification of the base and making suitable birational modifications over each stratum, can be achieved uniformly on all fibers. 
\end{remark}

\begin{lem}\label{adjoint terminal}
Let $0<\epsilon, \delta\leq 1$ be real numbers. Let $\fA:=(X,\cF,B,t)$ be an adjoint foliated surface. Assume that
    \begin{enumerate}
        \item $t\leq 1-\delta$,
        \item the coefficients of $B$ are at most $1-\epsilon$,
        \item $(X,B)$ is $\cF$-adapted log smooth, and
        \item $\cF$ is lc.
    \end{enumerate}
    Then $\fA$ is $\delta\epsilon$-lc. 
    If we furthermore assume that 
    \begin{enumerate}
        \item[(5)] the irreducible components of $\Supp B$ are pairwise disjoint, and
        \item[(6)] $\cF$ is $\delta_0$-canonical, where $0<\delta_0<\min\{\frac{1}{4},\frac{\epsilon\delta}{\epsilon\delta+1-\delta}\}$,
    \end{enumerate}
  then $\fA$ is terminal. 

\begin{proof}
\begin{enumerate}[label=\textsl{Step} \arabic*., wide=13pt, itemsep=13pt]
\item In this step, we prove that $\fA$ is $\delta\epsilon$-lc under assumptions~(1)--(4).

We first show that $(\cF,\nB)$ is log canonical. Indeed, near any closed point of $\Sing(\cF)$, the divisor $B^{\mathrm{ninv}}$ is empty, so
the assertion follows from the log canonicity of $\cF$. On the other hand, since $(X,B)$ is $\cF$-adapted log smooth, the pair $(\cF,B^{\mathrm{ninv}})$ is formally
locally toroidal away from $\Sing(\cF)$. Since all
the coefficients of $B^{\mathrm{ninv}}$ are at most $1$, it follows from \cite[Theorem~5.7]{CC25ToricToroidalFoliations} that
$(\cF,B^{\mathrm{ninv}})$ is log canonical.

Moreover, since $(X,B)$ is log smooth and the coefficients of $B$ are at
most $1-\epsilon$, the pair $(X,B)$ is $\epsilon$-lc. Hence, for every
prime divisor $E$ over $X$,
\begin{align*}
A(E,\fA)
&=tA(E,\cF,B^{\mathrm{ninv}})+(1-t)A(E,X,B) \\
&\geq (1-t)\epsilon
\geq \delta\epsilon.
\end{align*}
Therefore, $\fA$ is $\delta\epsilon$-lc.

\item In this step, we prove that $\fA$ is terminal under the additional assumptions~(5) and~(6).

Let $E$ be an exceptional divisor over $X$. Since the irreducible components of $\Supp B$ are pairwise disjoint and have coefficients strictly smaller than one, the
pair $(X,B)$ is terminal.

Suppose first that $E$ is $\cF$-invariant. Since
$(\cF,B^{\mathrm{ninv}})$ is log canonical, we have
\[a(E,\cF,B^{\mathrm{ninv}})\geq 0.\]
It follows that
\[a(E,\fA)=ta(E,\cF,B^{\mathrm{ninv}})+(1-t)a(E,X,B)>0.\]

Suppose now that $E$ is not $\cF$-invariant. Then
$P:=c_X(E)\in\Sing(\cF)$, and $B^{\mathrm{ninv}}$ is empty near $P$. Since $E$ is not $\cF$-invariant, the germ of $\cF$ at $P$ is dicritical, and hence not canonical at $P$ by \cite[Theorem~3.24]{CS26RecentProgress}. Since $\cF$ is $\delta_0$-canonical, it follows from \cite[Corollary~4.10 and Definition~4.6]{PS19EffectiveAlgebraicIntegration} that the germ is analytically conjugate
to the foliation generated by
\[px\frac{\partial}{\partial x}+qy\frac{\partial}{\partial y},\]
where $p,q$ are relatively prime positive integers, and $E$ is the unique
non-invariant exceptional divisor over $P$. Moreover, $E$ is extracted by
the weighted blow-up of weight $(p,q)$ and
\[a(E,X)=p+q-1\geq\frac{1-\delta_0}{\delta_0}.\]

Since the components of $\Supp B$ are pairwise disjoint, at most one component of $B$ passes through $P$. If such a component exists, then it is $\cF$-invariant and is a smooth separatrix. After interchanging $p$ and $q$ if necessary, we may assume that its order along $E$ is $p$. Thus,
\begin{align*}
a(E,X,B)
&\geq p+q-1-(1-\epsilon)p \\
&=\epsilon p+q-1 \\
&\geq\epsilon(p+q-1) \\
&\geq\frac{\epsilon(1-\delta_0)}{\delta_0}
>\frac{1-\delta}{\delta}.
\end{align*}
Since $\cF$ is log canonical, $a(E,\cF,B^{\mathrm{ninv}})=a(E,\cF)\geq-1$. Therefore,
\begin{align*}
a(E,\fA)
&\geq -t+(1-t)a(E,X,B) \\
&>-t+\frac{(1-t)(1-\delta)}{\delta} \\
&=\frac{1-\delta-t}{\delta}
\geq 0.
\end{align*}
Hence $\fA$ is terminal.
\end{enumerate}
\end{proof}
\end{lem}

\subsection{Types of models}
We recall some standard terminology concerning birational models of divisors. Our conventions follow \cite[\S3.6]{BCHM10} and \cite[\S2]{HX13ExistenceLogCanonicalClosure}. 

\begin{definition}
Let $f:X\dashrightarrow Y$ be a birational map between two normal projective varieties which does not extract any divisor. Let $D$ be an $\bR$-Cartier divisor on $X$ such that $D_Y:=f_*D$ is $\bR$-Cartier.

We say that $f$ is \emph{$D$-non-positive} if there exists a common
resolution $p:W\to X$ and $q:W\to Y$ of $f$ such that
\[p^*D=q^*D_Y+E,\]
where $E\geq 0$ is $q$-exceptional.

We say that $f$ is \emph{$D$-negative} if it is $D$-non-positive and $\Supp E$ contains the strict transform of every $f$-exceptional divisor.
\end{definition}

\begin{definition}
Let $f:X\dashrightarrow Y$ and $D$ be as above.

We say that $Y$ is a \emph{minimal model} of $D$ if $Y$ is
$\bQ$-factorial, $f$ is $D$-negative, and $D_Y$ is nef. If, in addition, $D_Y$ is semi-ample, then we say that $Y$ is a
\emph{good minimal model} of $D$.
\end{definition}

\begin{definition}
Let $D$ be an $\bR$-Cartier divisor on a normal projective variety $X$,
and let $g:X\dashrightarrow Z$ be a rational map to a normal projective
variety.

We say that $Z$ is an \emph{ample model} of $D$ if there exist a good minimal model $f:X\dashrightarrow Y$ of $D$, a contraction morphism $h:Y\to Z$, and an ample $\bR$-Cartier divisor $H$ on $Z$ such that
$g=h\circ f$ and $f_*D\sim_{\bR}h^*H$.
\end{definition}

\begin{definition} 
Let $\fA:=(X,\cF,B,t)$ be an adjoint foliated structure. Let $f:X\dashrightarrow Y$ be a birational map between two normal projective varieties which does not extract any divisor. We say that $Y$ is a \emph{minimal model} (resp. \emph{good minimal model}) of $\fA$ if $Y$ is a minimal model (resp. good minimal model) of $\KFA$. Let $g:X\dashrightarrow Z$ be a rational map to a normal projective variety. We say that $Z$ is the \emph{ample model} of $\fA$ if $Z$ is the ample model of $\KFA$. If $\KFA$ is big, then the ample model of $\fA$ is also called the \emph{log canonical model} of $\fA$. 
\end{definition}

\subsection{Boundedness}

We recall the notions of boundedness and log boundedness for adjoint foliated structures; see \cite[\S2.19]{Bir19AntiPluricanonicalSystemsFano} and \cite[Definition~3.31]{Cas+25OnFiniteGenerationAndBoundedness}.

\begin{definition}\label{def of boundedness}
Let $\cP$ be a set of adjoint foliated structures
$\fA=(X,\cF,B,t)$. We say that $\cP$ is
\emph{log birationally bounded} (resp. \emph{log bounded})
if there exist finitely many flat projective morphisms
$f_i:\cX_i\to Z_i$, $i=1,\dots,N$, between normal varieties
with normal fibers, reduced divisors $\cE_i$ on $\cX_i$
containing no fiber, and foliations $\cG_i$ on $\cX_i$
satisfying the following:
\begin{enumerate}
\item For every closed point $z\in Z_i$, we have
\[T_{\cX_i/Z_i}|_{\cX_{i,z}}\simeq T_{\cX_{i,z}}.\]
\item We have $T_{\cG_i}\subseteq T_{\cX_i/Z_i}$, and for
every closed point $z\in Z_i$, the inclusion
\[T_{\cG_i}|_{\cX_{i,z}}\subseteq T_{\cX_{i,z}}\]
defines a foliation $\cF_{i,z}$ on $\cX_{i,z}$.
\item For every $\fA=(X,\cF,B,t)\in\cP$, there exist
$i\in\{1,\dots,N\}$, a closed point $z\in Z_i$, and a
birational map (resp. an isomorphism)
$\phi:\cX_{i,z}\dashrightarrow X$ such that
$\phi_*\cF_{i,z}=\cF$, and $\cE_{i,z}$ contains the support
of $\phi^{-1}_*B$ and every $\phi$-exceptional divisor
(resp. $\cE_{i,z}$ coincides with the support of
$\phi^{-1}_*B$).
\end{enumerate}

When $t=0$ for every $\fA=(X,\cF,B,t)\in\cP$, we simply say that the pairs $(X,B)$ form a log birationally bounded (resp. log bounded) family.
\end{definition}

The following proposition is an easy consequence of \cite[Proposition 3.36]{Cas+25OnFiniteGenerationAndBoundedness}, which provides a useful criterion for log boundedness of adjoint foliated structures in terms of uniform bounds on degrees with respect to a very ample divisor.

\begin{prop}\label{bdd criterion}
Let $d,r$ be positive integers, and let $\cP$ be a set of
$d$-dimensional adjoint foliated structures. Assume that for every $\fA=(X,\cF,B,t)\in\cP$, there exists a very ample divisor $A$ on $X$ such that
\begin{enumerate}
\item $A^d\leq r$,
\item $A^{d-1}\cdot\Supp B\leq r$, and
\item $A^{d-1}\cdot K_\cF\leq r$.
\end{enumerate}
Then $\cP$ is log bounded.
\end{prop}

\begin{proof}
We follow the proof of \cite[Theorem 2.4.1]{Cas+25OnFiniteGenerationAndBoundedness}.

Take $\fA=(X,\cF,B,t)\in \cP$. By \cite[Lemma 2.20]{Bir19AntiPluricanonicalSystemsFano}, $(X,\Supp(B+A))$ belongs to a bounded family. Hence, by the definition of bounded family and by Serre's property (cf. \cite[Theorem 1.2.6]{Laz04PositivityAlgebraicGeometryI}), there exists a positive integer $s$ depending
only on $d$ and $r$, such that $\Omega_X^m(sA)$ is globally generated for every $1\leq m\leq d$.

Let $m=\rank\cF$. By \cite[\S2.2]{CS25FoliationAdjunction}, the Pfaff field associated with $\cF$ gives a generically surjective morphism
\[\Omega_X^m\longrightarrow\cO_X(K_\cF).\]
It follows that
\[H^0\bigl(X,\cO_X(K_\cF+sA)\bigr)\neq 0.\]
Hence, there exists an effective integral divisor
$D\in|K_\cF+sA|$. We have
\[A^{d-1}\cdot D=A^{d-1}\cdot K_\cF+sA^d\leq (s+1)r.\]
Applying \cite[Lemma~2.20]{Bir19AntiPluricanonicalSystemsFano} again,
we see that the couples
$(X,\Supp(B+D))$ are log bounded.

After stratification, we may assume that $A$ and $D$ are induced by divisors $\cA_i$ and $\cD_i$ on the corresponding total spaces $\cX_i$. Since
\[D\sim K_\cF+sA,\]
we have
\[\cO_X(-K_\cF)\simeq\cO_X(sA-D).\]
Therefore, the rank one reflexive sheaves $\cO_{\cX_i}(s\cA_i-\cD_i)$ restrict to $\cO_X(-K_\cF)$ on the corresponding fibers.

We may now apply the proof of
\cite[Proposition~3.36]{Cas+25OnFiniteGenerationAndBoundedness}
to conclude that the foliated pairs $(X,\cF)$ form a bounded family. Although that proposition is stated for a $\bQ$-Cartier divisor on the total space, its proof only uses the corresponding rank one reflexive sheaf, and hence the same argument applies here. Since the supports of $B$ also belong to a bounded family, it follows that $\cP$ is log bounded.
\end{proof}

\section{MMP for adjoint foliated surfaces}\label{Section:MMP}
In this section, we prove Theorem~\ref{adjoint MMP}, that is, one can run a $\KFA$-MMP for klt adjoint foliated surfaces. We first establish the cone theorem and contraction results for $\KFA$-negative extremal rays, and then use them to prove the existence and termination of the $\KFA$-MMP.

\begin{thm}\label{adjoint MMP}
    Let $\fA:=(X,\cF,B,t)$ be a klt adjoint foliated surface. Then we may run a $\KFA$-MMP $f:X\to Y$. Moreover, the following properties hold:
    \begin{enumerate}
        \item $\fB:=(Y,f_*\cF,f_*B,t)$ is klt.
        \item If $\KFA$ is pseudo-effective, then $\KFB$ is nef.
        \item If $\KFA$ is not pseudo-effective, then $Y$ admits a fibration $Y\to Z$ such that $\rho(Y/Z)=1$ and $-\KFB$ is ample over $Z$.
    \end{enumerate}
\end{thm}

\subsection{Cone theorem}
We begin by proving the cone theorem for klt adjoint foliated surfaces. The main point is that every $\KFA$-negative extremal ray is negative either with respect to the underlying pair or with respect to the foliated pair. Thus the classical cone theorem and the cone theorem for foliated surfaces can be combined to control the $\KFA$-negative part of the cone of curves.

\begin{thm}\label{conethm}
Let $\fA:=(X,\cF,B,t)$ be a klt adjoint foliated surface.
Let $\{R_i\}_{i\in I}$ be the set of all $\KFA$-negative extremal rays. Then
\begin{enumerate}
\item We have \[\overline{\textnormal{NE}}(X)=\overline{\textnormal{NE}}(X)_{\KFA\geq 0}+\sum_{i\in I}R_i.\]

\item Each $R_i$ is spanned by a rational curve $C_i$ with $0<-\KFA\cdot C_i\leq 4$.

\item For any ample divisor $A$ on $X$, the set $I_A:=\{i\in I\mid (\KFA+A)\cdot C_i<0\}$ is finite. In particular, $I$ is countable. Moreover, we have \[\overline{\textnormal{NE}}(X)=\overline{\textnormal{NE}}(X)_{\KFA+A\geq 0}+\sum_{i\in I_A}R_i.\]
\end{enumerate}

\begin{proof}
Since $\fA$ is klt, by Proposition~\ref{control singularities}, $(X,B)$ is a klt pair. In particular, $X$ is klt and hence $\bQ$-factorial. Therefore, $K_X+B$ and $\KCF+\nB$ are $\bR$-Cartier.

For (1), fix an ample divisor $H$ on $X$. Since $H$ is strictly positive on $\overline{\textnormal{NE}}(X)\setminus\{0\}$, the set \[\Sigma:=\{\alpha\in\overline{\textnormal{NE}}(X)\mid H\cdot\alpha=1\}\] is compact and convex. By Minkowski's theorem, every element of $\Sigma$ is a convex combination of finitely many extreme points of $\Sigma$. These extreme points correspond precisely to the extremal rays of $\overline{\textnormal{NE}}(X)$. Therefore, every element of $\overline{\textnormal{NE}}(X)$ is a finite sum of classes lying on extremal rays.

We may group together the summands lying on $\KFA$-non-negative extremal rays. Their sum belongs to
$\overline{\textnormal{NE}}(X)_{\KFA\geq 0}$, while every remaining summand lies on some $R_i$. Hence
\[\overline{\textnormal{NE}}(X)=\overline{\textnormal{NE}}(X)_{\KFA\geq 0}+\sum_{i\in I}R_i.\]

For (2), let $R_i$ be a $\KFA$-negative extremal ray. Since
\[\KFA=t(\KCF+\nB)+(1-t)(K_X+B),\]
the ray $R_i$ is either $(K_X+B)$-negative or
$(\KCF+\nB)$-negative. By
\cite[Theorem~18.2]{Fuj11FundamentalTheoremsLogMMP} and
\cite[Theorem~6.3]{Spi20HigherDimensionalFoliatedMori}, we may choose
a rational curve $C_i$ spanning $R_i$ such that
\[(K_X+B)\cdot C_i\geq -4
\quad\text{and}\quad
(\KCF+\nB)\cdot C_i\geq -4.\]
Therefore,
\[-\KFA\cdot C_i=-t(\KCF+\nB)\cdot C_i-(1-t)(K_X+B)\cdot C_i\leq 4.\]

For (3), let
\[I_{A,X}:=\{i\in I\mid (K_X+B+A)\cdot C_i<0\}\] and
\[I_{A,\cF}:=\{i\in I\mid (\KCF+\nB+A)\cdot C_i<0\}.\]
Since \[\KFA+A=t(\KCF+\nB+A)+(1-t)(K_X+B+A),\]
we have \[I_A\subseteq I_{A,X}\cup I_{A,\cF}.\]
By \cite[Theorem 16.6]{Fuj11FundamentalTheoremsLogMMP}
and \cite[Theorem 6.3]{Spi20HigherDimensionalFoliatedMori},
the sets $I_{A,X}$ and $I_{A,\cF}$ are finite. Hence $I_A$ is finite. Moreover,
\[I=\bigcup_{n\in\bN_{>0}}I_{\frac{1}{n}A},\]
so $I$ is countable. Finally, if $i\notin I_A$, then
$(\KFA+A)\cdot R_i\geq 0$. Therefore, the last equality follows
from (1).
\end{proof}
\end{thm}

\subsection{Divisorial contractions}
We next construct contractions of $\KFA$-negative extremal rays in the birational case. If such a ray is negative with respect to $K_X+B$, then the contraction follows from the classical contraction theorem. Thus the main new case is when the ray is negative with respect to $\KCF+\nB$. In this subsection, we treat the divisorial case.

\begin{thm}\label{divisorial contraction: foliation negative}
    Assume that 
\begin{enumerate}
    \item $\fA:=(X,\cF,B,t)$ is a klt adjoint foliated surface,
    \item $R$ is a $\KFA$-negative and $(\KCF+\nB)$-negative extremal ray, and
    \item $A$ is an ample $\bR$-divisor such that $H:=\KFA+A$ is nef and big, and $H$ vanishes exactly on $R$. 
\end{enumerate}
Then $R$ is spanned by an invariant rational curve $C$, and there is a divisorial contraction $f:X\to Y$ which contracts only $C$.

\begin{proof}
\begin{enumerate} [label=\textsl{Step} \arabic{enumi}., wide=13pt, itemsep=13pt]
\item In this step, we construct an F-dlt modification of $(\cF,\nB)$ and introduce some notation.

By \cite[Theorem 6.3]{Spi20HigherDimensionalFoliatedMori}, $R$ is spanned by an invariant rational curve $C$. By \cite[Theorem 8.1]{CS21MMPCorankOneThreefolds}, $(\cF,\nB)$ admits an F-dlt modification $\pi: W\to X$ such that $W$ is klt, and hence $\bQ$-factorial, and the pullback foliation $\cG$ is non-dicritical. Let $B_W$ and $C_W$ be the strict transform of $B$ and $C$ on $W$, respectively, and let $E=\sum E_i$ be the sum of the reduced $\pi$-exceptional divisors. Let \[\fB:=(W,\cG,B_W+E,t)\] be an adjoint foliated surface on $W$. Since $\fA$ is klt, we may write 
\begin{equation}\tag{$\star$}
\KFB=\pi^*\KFA+\sum \tau_iE_i,
\end{equation} where $\tau_i>0$ for every $i$.

\item In this step, we contract $C_W$ and obtain a new birational model $V$ with $\mu: W\to V$.

Since $\pi^*H$ is nef and big and 
\[\pi^*H\cdot C_W=H\cdot C=0,\] 
the Hodge index theorem implies that $C_W^2<0$. Moreover, by the definition of F-dlt modification, we have 
\begin{align*}
    (\KCG+\nB_W+\nE)\cdot C_W&\leq \pi^*(\KCF+\nB)\cdot C_W\\
    &=(\KCF+\nB)\cdot C<0.
\end{align*}
Therefore, the numerical class of $C_W$ spans a $(\KCG+\nB_W+\nE)$-negative extremal ray. 

Since $(\cG,\nB_W+\nE)$ is F-dlt and $\cG$ is non-dicritical, it follows from \cite[Proposition 5.1(1)]{CS26RecentProgress} that $(W,B_W+E)$ is lc. Therefore, by
\cite[Theorem~2.8]{SS23EffectiveGenerationFoliatedSurfaces}, there
exists a $(\KCG+\nB_W+\nE)$-negative extremal contraction
$\mu:W\to V$ which contracts $C_W$.

Denote the pushforwards of $\cG$, $B_W$, $E$, and $E_i$ by $\cH$,
$B_V$, $E_V$, and $E_{i,V}$, respectively, and set
\[H_V:=\mu_*\pi^*H.\]
By the negativity lemma, $(\cH,\nB_V+\nE_V)$ is lc. Moreover, since
$\cG$ is non-dicritical and $\mu$ only contracts the invariant curve
$C_W$, the induced foliation $\cH$ is also non-dicritical.

\item In this step, we run a partial $(\KCH+\nB_V+\nE_V)$-MMP and obtain a new birational model $T$ with $\nu: V\to T$. 

Let $l$ be a sufficiently large positive number, and let 
\[M_V:=\KCH+\nB_V+\nE_V+lH_V.\]
If $M_V$ is nef, then we stop and define $\nu:V\to T$ to be the identity. In the following we assume that $M_V$ is not nef.

Let $R'$ be an $M_V$-negative extremal ray. Note that on surfaces, the pushforward of a nef $\bR$-divisor is still nef. Hence $H_V$ is nef on $V$. Since $l$ is sufficiently large, by boundedness of the length of extremal rays (cf. \cite[Theorem 6.3]{Spi20HigherDimensionalFoliatedMori}), $R'$ is an $H_V$-trivial extremal ray. Therefore, $R'$ is spanned by $E_{i_0,V}$ for some $i_0$, and 
\[(\KCH+\nB_V+\nE_V)\cdot E_{i_0,V}<0.\]
It follows from \cite[Theorem 6.3]{Spi20HigherDimensionalFoliatedMori} that $E_{i_0,V}$ is $\cH$-invariant.

Applying \cite[Theorem 2.8]{SS23EffectiveGenerationFoliatedSurfaces}, we may contract $E_{i_0,V}$ by a $(\KCH+\nB_V+\nE_V)$-negative extremal contraction. Repeating the process, we run an $M_V$-MMP, which is also a partial $(\KCH+\nB_V+\nE_V)$-MMP, and obtain a birational morphism $\nu:V\to T$ such that 
\[M_T:=\nu_*M_V\] is nef. 

Denote the pushforwards of $\cH$, $B_V$, $E_V$, $E_{i,V}$, and $H_V$
by $\cJ$, $B_T$, $E_T$, $E_{i,T}$, and $H_T$, respectively. We also set
\[A_T:=\nu_*\mu_*\pi^*A\]
and define
\[\fC:=(T,\cJ,B_T+E_T,t).\]

\item In this step, we run a partial $(K_T+B_T+E_T)$-MMP and obtain the desired birational model $Y$ with $\rho: T\to Y$.

By the negativity lemma, $(\cJ,\nB_T+\nE_T)$ is lc. 
Since $\cH$ is non-dicritical and $\nu$ only contracts invariant curves, the induced foliation $\cJ$ is also non-dicritical. Thus, it follows from \cite[Proposition 5.1(1)]{CS26RecentProgress} that $(T,B_T+E_T)$ is lc. 
Set
\[N_T:=\KFC+A_T+lH_T.\]
If $N_T$ is nef, then we stop and define $\rho:T\to Y$ to be the identity. In the following, we assume that $N_T$ is not nef.

Let $R''$ be an $N_T$-negative extremal ray. We have
\begin{align*}
N_T&=t(\KCJ+\nB_T+\nE_T)+(1-t)(K_T+B_T+E_T)+lH_T+A_T\\
&=(1-t)(K_T+B_T+E_T+lH_T)+tM_T+A_T.
\end{align*}
Since $M_T$ and $A_T$ are nef, the ray $R''$ is $(K_T+B_T+E_T+lH_T)$-negative. Since $l$ is sufficiently large, the boundedness of the length of extremal rays implies that $R''$ is $H_T$-trivial. Therefore, $R''$ is spanned by $E_{i_1,T}$ for some $i_1$, and
\[(K_T+B_T+E_T)\cdot E_{i_1,T}<0.\]
Take a sufficiently small positive number $\epsilon$ such that 
\[(K_T+(1-\epsilon)(B_T+E_T))\cdot E_{i_1,T}<0.\] 
Since $T$ is a klt surface and $(T,B_T+E_T)$ is lc, the pair 
\[(T,(1-\epsilon)(B_T+E_T))\]
is klt. Hence, by \cite[Theorem~3.7]{KM98BirationalGeometry}, we may
contract $E_{i_1,T}$ by a
$\bigl(K_T+(1-\epsilon)(B_T+E_T)\bigr)$-negative extremal contraction.

Repeating the process, we can run an $N_T$-MMP, which is also a partial $(K_T+B_T+E_T)$-MMP, and obtain a birational morphism $\rho:T\to Y$. Note that during this process, the foliated pair $(\cJ,\nB_T+\nE_T)$ may cease
to be lc; however, this condition is no longer needed in the remainder of the proof.

Let $H_Y$ and $E_{i,Y}$ be the pushforwards of $H_T$ and $E_{i,T}$. Using $(\star)$, we obtain
\begin{align*}
    N_Y&:=\rho_*N_T\\
    &=\rho_*(\KFC+A_T+lH_T)\\
    &=\rho_*(\nu_*\mu_*\pi^*\KFA+\sum\tau_iE_{i,T}+A_T+lH_T)\\
    &=(l+1)H_Y+\sum\tau_iE_{i,Y}.
\end{align*}
By construction, $N_Y$ is nef.

Recall that all contractions constructed above are $\pi^*H$-trivial. Therefore, \[H_Y\cdot(\sum\tau_iE_{i,Y})=\pi^*H\cdot\sum\tau_iE_i=0.\]
It follows that 
\[\Big(\sum\tau_iE_{i,Y}\Big)^2=N_Y\cdot \Big(\sum\tau_iE_{i,Y}\Big)\geq 0.\]
Since $H_Y$ is nef and big, the Hodge index theorem gives
\[\sum\tau_iE_{i,Y}\equiv0.\]
As $\tau_i>0$ for every $i$, all the divisors $E_i$ are contracted by the morphism $W\to Y$. Consequently, the induced birational map descends to a morphism $f:X\to Y$
which contracts only $C$. This completes the proof.
\end{enumerate}
\end{proof}
\end{thm}

\subsection{Fiber type contractions}
We next treat the case of fiber type. The argument is different from the divisorial case: when the supporting divisor is nef but not big, the bend-and-break method produces a moving family of rational curves spanning the given extremal ray. We then show that this ray is also negative with respect to $K_X+B$, so that the desired contraction follows from the classical contraction theorem.

\begin{thm}\label{Mori fiber space:foliation negative}
    Assume that 
\begin{enumerate}
    \item $\fA:=(X,\cF,B,t)$ is a klt adjoint foliated surface,
    \item $R$ is a $\KFA$-negative extremal ray, and
    \item $A$ is an ample $\bR$-divisor such that $H:=\KFA+A$ is nef and not big, and $H$ vanishes exactly on $R$.
\end{enumerate}
Then $R$ is $(K_X+B)$-negative. In particular, there is a fibration $f:X\to Z$ such that for any irreducible curve $\Sigma$, $f(\Sigma)$ is a point if and only if $[\Sigma]\in R$.

\begin{proof}
Assume for contradiction that
\[(K_X+B)\cdot R\geq 0.\]
Since $K_{\fA}\cdot R<0$, it follows that
\[(K_{\cF}+B^{\ninv})\cdot R<0.\]

\begin{enumerate}[label=\textsl{Step} \arabic{enumi}., wide=13pt, itemsep=13pt]
\item In this step, we prove that through a general point of $X$, there is an invariant rational curve whose numerical class spans $R$.

We first consider the case $H\equiv0$. Since $H$ vanishes exactly on $R$, we have
\[\overline{\textnormal{NE}}(X)=R.\]
Let $L$ be an ample Cartier divisor and let $\Gamma\in|mL|$ be general for $m\gg0$. Then
\[(K_{\cF}+B^{\ninv})\cdot\Gamma<0,\]
and hence $K_{\cF}\cdot\Gamma<0$. By \cite[Theorem~2.27]{Spi20HigherDimensionalFoliatedMori}, varying $\Gamma$, we obtain a moving family of invariant rational curves whose numerical classes span $R$.

We may therefore assume that $H\not\equiv0$. Since $H$ is nef and not big, we have $H^2=0$. Moreover,
\[K_{\fA}\cdot H=-A\cdot H<0.\]
Therefore, either $(K_{\cF}+B^{\ninv})\cdot H<0$ or
$(K_X+B)\cdot H<0$.

In the first case, applying \cite[Corollary~2.28]{Spi20HigherDimensionalFoliatedMori} with $D_1=D_2=H$, we obtain a moving family of invariant rational curves $C$ such that $H\cdot C=0$. Thus $[C]\in R$.

In the second case, we apply \cite[Theorem 6.1]{KMM94LogAbundanceTheoremThreefolds} with $D_1=D_2=H$. Thus, through a general point of $X$, there is a rational curve $C_0$ such that $H\cdot C_0=0$, and hence $[C_0]\in R$.

Since $(K_{\cF}+B^{\ninv})\cdot R<0$, we may choose an ample $\bR$-divisor $A'$ such that
\[H':=K_{\cF}+B^{\ninv}+A'\]
is nef and vanishes exactly on $R$. Since the curves $C_0$ form a moving family and $H'\cdot C_0=0$, the divisor $H'$ is not big. Hence $(H')^2=0$, and
\[(K_{\cF}+B^{\ninv})\cdot H'=-A'\cdot H'<0.\]
Applying \cite[Corollary~2.28]{Spi20HigherDimensionalFoliatedMori} with $D_1=D_2=H'$, we obtain a moving family of invariant rational curves $C$ such that $H'\cdot C=0$. Thus $[C]\in R$.

\item In this step, we derive a contradiction.

Let $C$ be a general member of the moving family constructed in Step~1. Since $\fA$ is klt, Proposition~\ref{control singularities} implies that $(X,B)$ is a klt pair. By \cite[Theorem~8.1]{CS21MMPCorankOneThreefolds}, $(\cF,\nB)$ admits an F-dlt modification $\pi:W\to X$ such that $W$ is $\bQ$-factorial and the pullback foliation $\cG$ is non-dicritical.

Let $B_W$ and $C_W$ be the strict transforms of $B$ and $C$ on $W$, respectively, and let $E=\sum E_i$ be the sum of reduced $\pi$-exceptional divisors. Then
\[(K_{\cG}+\nB_W+\nE)\cdot C_W\leq\pi^*(K_{\cF}+\nB)\cdot C_W=(K_{\cF}+\nB)\cdot C<0.\]

Since $\cG$ is non-dicritical and the curves $C_W$ form a moving family of invariant rational leaves, $\cG$ is induced by a fibration $h:W\to T$, and a general $C_W$ is a general fiber of $h$. Therefore,
\[K_{\cG}\cdot C_W=-2=K_W\cdot C_W.\]

Since $(X,B)$ is klt, we may write
\[K_W+B_W+E=\pi^*(K_X+B)+P,\]
where $P\geq0$ is $\pi$-exceptional. Moreover, a general fiber $C_W$ does not intersect $B_W^{\inv}$ or $E^{\inv}$. Hence
\[(K_X+B)\cdot C\leq(K_W+B_W+E)\cdot C_W=(K_{\cG}+B_W^{\ninv}+E^{\ninv})\cdot C_W<0.\]
This contradicts the assumption that $(K_X+B)\cdot R\geq0$. Therefore, $R$ is $(K_X+B)$-negative. By \cite[Theorem~3.7]{KM98BirationalGeometry}, there is a fibration $f:X\to Z$ such that for any irreducible curve $\Sigma$, $f(\Sigma)$ is a point if and only if $[\Sigma]\in R$.
\end{enumerate}
\end{proof}
\end{thm}

\subsection{Adjoint MMP}
We are now ready to prove Theorem~\ref{adjoint MMP}.

\begin{proof}
We construct the $\KFA$-MMP inductively. Set
\[\fA_0:=\fA=(X_0,\cF_0,B_0,t).\]
Suppose that a klt adjoint foliated surface
\[\fA_i:=(X_i,\cF_i,B_i,t)\]
has been constructed.

If $K_{\fA_i}$ is nef, then we stop. Otherwise, let $R_i$ be a
$K_{\fA_i}$-negative extremal ray. Since
\[K_{\fA_i}=t(K_{\cF_i}+\nB_i)+(1-t)(K_{X_i}+B_i),\]
the ray $R_i$ is either $(K_{X_i}+B_i)$-negative or
$(K_{\cF_i}+\nB_i)$-negative.

If $(K_{X_i}+B_i)\cdot R_i<0$, then $R_i$ admits an extremal
contraction by the classical contraction theorem
\cite[Theorem~3.7]{KM98BirationalGeometry}.

Assume instead that
\[(K_{\cF_i}+\nB_i)\cdot R_i<0.\]
By Theorem~\ref{conethm}, the $K_{\fA_i}$-negative extremal rays
are locally finite in the $K_{\fA_i}$-negative part of
$\overline{\textnormal{NE}}(X_i)$. It follows that $R_i$ is exposed in the sense of
\cite[Definition~6.1]{Spi20HigherDimensionalFoliatedMori}. Thus, there
exists a nef $\bR$-divisor $H_i$ such that
\[\overline{\textnormal{NE}}(X_i)\cap H_i^\perp=R_i.\]
Since $K_{\fA_i}\cdot R_i<0$, for a sufficiently large positive real number $c$, the divisor
\[A_i:=cH_i-K_{\fA_i}\]
is ample. Replacing $H_i$ by $cH_i$, we may therefore write
\[H_i=K_{\fA_i}+A_i,\]
where $A_i$ is ample and $H_i$ vanishes exactly on $R_i$.
If $H_i$ is big, then
Theorem~\ref{divisorial contraction: foliation negative} gives a
divisorial contraction of $R_i$. If $H_i$ is not big, then
Theorem~\ref{Mori fiber space:foliation negative} gives a fiber type contraction of $R_i$.

Thus, in all cases, there exists an extremal contraction
\[\mu_{i+1}:X_i\to T_i\]
associated with $R_i$. If $\mu_{i+1}$ is birational, we set
\[X_{i+1}:=T_i,\qquad \fA_{i+1}:=(\mu_{i+1})_*\fA_i.\]
By the negativity lemma, $\fA_{i+1}$ is again klt, so we may continue the process. If $\mu_{i+1}$ is of fiber type, we stop and set
\[Y:=X_i,\qquad Z:=T_i.\]

Each birational morphism decreases the Picard number by one. Therefore, after finitely many steps, the process terminates with one of the following two outcomes:
\begin{enumerate}[label=(\roman*)]
\item a klt adjoint foliated surface $\fB$ such that $\KFB$ is nef;
\item a fiber type contraction $Y\to Z$ such that $\rho(Y/Z)=1$ and $\KFB$ is negative over $Z$, where $\fB$ is the induced adjoint foliated structure on $Y$.
\end{enumerate}
Let $f:X\to Y$ be the composition of the birational contractions. This proves that $\fB$ is klt and, in the fiber type case, that $-\KFB$ is ample over $Z$.

It remains to distinguish the two possible outputs. If $\KFA$ is pseudo-effective, then its pushforward on every birational model appearing in the MMP is pseudo-effective. Hence the MMP cannot end with a fiber type contraction, and therefore $\KFB$ is nef.

If $\KFA$ is not pseudo-effective, then, by the negativity lemma, the adjoint canonical divisor remains non-pseudo-effective after every birational step. Since a nef divisor is pseudo-effective, the MMP cannot terminate with a nef model. It must therefore end with a fiber type contraction $Y\to Z$ such that $\rho(Y/Z)=1$ and
$-\KFB$ is ample over $Z$.
\end{proof}

As in the classical case (cf. \cite[\S 3.6-7]{KM98BirationalGeometry}), the proof of Theorem \ref{adjoint MMP} can be adapted to yield the following relative version: 

\begin{cor}\label{relative adjoint MMP}
     Let $\fA:=(X,\cF,B,t)$ be a klt adjoint foliated surface. Let $\pi:X\to U$ be a projective morphism. Then we may run a $\KFA$-MMP $f:X\to Y$ over $U$. Moreover, setting $\fB:=(Y,f_*\cF,f_*B,t)$, the following properties hold:
    \begin{enumerate}
        \item $\fB$ is klt.
        \item If $\KFA$ is pseudo-effective over $U$, then $\KFB$ is nef over $U$.
        \item If $\KFA$ is not pseudo-effective over $U$, then $Y$ admits a fibration $Y\to Z$ over $U$ such that $\rho(Y/Z)=1$ and $-\KFB$ is ample over $Z$.
    \end{enumerate}
\end{cor}

\subsection{Good minimal models}
We next prove the existence and uniqueness of good minimal models for klt adjoint foliated surfaces whose boundaries contain ample divisors.

\begin{lem}\label{nef but not ample}
Assume that 
\begin{enumerate}
    \item $\fA:=(X,\cF,B,t)$ is a klt adjoint foliated surface,
    \item $\nB\geq A$, where $A\geq 0$ is an ample $\bR$-divisor, and
    \item $\KFA$ is nef but not ample.
\end{enumerate}
Then there exists a $\KFA$-trivial extremal ray $R$. 

\begin{proof}
Since $\fA$ is klt, it follows from Proposition \ref{control singularities} that $(X,B)$ is a klt pair. Set \[\fB:=(X,\cF,B-\frac{1}{2}A,t).\] By assumption, $\KFB$ is not nef, otherwise $\KFA=\KFB+\frac{1}{2}A$ is ample. 
By Theorem~\ref{conethm}, there are only finitely many $\KFB$-negative extremal rays, say $R_1,\dots,R_m$.

Assume that there is no $\KFA$-trivial extremal ray. Then $\KFA\cdot R_i>0$ for $1\leq i\leq m$. Thus, $(\KFA-\epsilon A)\cdot R_i>0$ for a sufficiently small positive real number $\epsilon$. It then follows that $\KFA-\epsilon A$ is nef. Hence $\KFA$ is ample, which is a contradiction. Therefore, there exists a $\KFA$-trivial extremal ray.
\end{proof}
\end{lem}

\begin{thm}\label{good minimal models}
Let $\fA:=(X,\cF,B,t)$ be a klt adjoint foliated surface such that $\KFA$ is pseudo-effective. Assume that $\nB\geq A$ for some ample $\bR$-divisor $A\geq 0$ on $X$. Let $f:X\to Y$ be the output of a $\KFA$-MMP.  Then $Y$ is the unique good minimal model of $\fA$.

\begin{proof}
Let $\fB:=(Y,\cG,B_Y,t)$ be the pushforward of $\fA$, and let $A_Y:=f_*A$. Since $f$ is the output of a $\KFA$-MMP, $\fB$ is klt and $\KFB$ is nef. Moreover, $B_Y^{\ninv}\geq A_Y$, where $A_Y$ is ample. The uniqueness of the good minimal model follows from the uniqueness of the Zariski decomposition. Thus, it remains to prove that $\KFB$ is semi-ample.

Set $Y_0:=Y$, $\fB_0:=\fB$, and $A_0:=A_Y$. We inductively construct a sequence of contractions
\[Y=Y_0\xrightarrow{\mu_1}Y_1\xrightarrow{\mu_2}\cdots\xrightarrow{\mu_n}Y_n.\]
Suppose that $Y_i$, $\fB_i:=(Y_i,\cG_i,B_i,t)$, and $A_i$ have been constructed such that $\fB_i$ is klt, $K_{\fB_i}$ is nef, $B_i^{\ninv}\geq A_i$, and $A_i$ is ample.

If $K_{\fB_i}$ is ample, then we stop. Otherwise, by Lemma~\ref{nef but not ample}, there exists a $K_{\fB_i}$-trivial extremal ray $R_i$. Set
\[\fB_i':=(Y_i,\cG_i,B_i-\frac12A_i,t).\]
Then
\[K_{\fB_i'}\cdot R_i=-\frac12A_i\cdot R_i<0.\]
As in the proof of Theorem~\ref{adjoint MMP}, there exists a contraction
\[\mu_{i+1}:Y_i\to Y_{i+1}\]
which contracts exactly the curves whose numerical classes belong to $R_i$.

If $\dim Y_{i+1}<2$, then $K_{\fB_i}$ is semi-ample and we stop. Otherwise, $\mu_{i+1}$ is birational. Let
\[\fB_{i+1}:=(\mu_{i+1})_*\fB_i\quad\text{and}\quad A_{i+1}:=(\mu_{i+1})_*A_i.\]
Since $R_i$ is $K_{\fB_i}$-trivial, we have
\[K_{\fB_i}=\mu_{i+1}^*K_{\fB_{i+1}}.\]
In particular, $\fB_{i+1}$ is klt. Moreover, $B_{i+1}^{\ninv}\geq A_{i+1}$ and $A_{i+1}$ is ample. Thus, the construction can be continued.

Each birational morphism decreases the Picard number by one. Hence, after finitely many steps, the process terminates either with a klt adjoint foliated surface $\fB_n$ such that $K_{\fB_n}$ is ample, or with a $K_{\fB_n}$-trivial contraction $Y_n\to Y_{n+1}$ such that $\dim Y_{n+1}<2$. In either case, $K_{\fB_n}$ is semi-ample. Since
\[\KFB=\mu_1^*\cdots\mu_n^*K_{\fB_n},\]
it follows that $\KFB$ is semi-ample. Therefore, $Y$ is the unique good minimal model of $\fA$.
\end{proof}
\end{thm}

We next prove the existence and uniqueness of the log canonical model for klt adjoint foliated surfaces of general type.

\begin{cor}\label{ample model}
    Let $\fA:=(X,\cF,B,t)$ be a klt adjoint foliated surface such that $\KFA$ is big. Let $f:X\to Y$ be the output of a $\KFA$-MMP. Define $\fB:=(Y,\cG,B_Y,t)$ to be the pushforward of $\fA$. Then there exists a birational morphism $g:Y\to Z$ such that if we take $\fC$ to be the pushforward of $\fB$, then we have 
    \begin{enumerate}
        \item $\fC$ is klt,
        \item $\KFB=g^*\KFC$, and
        \item $\KFC$ is ample.
    \end{enumerate}
Moreover, $Z$ is uniquely determined.

\begin{proof}
Since $\KFA$ is big, $\KFB$ is nef and big. Hence, we can choose an effective ample $\bR$-divisor $A$ and an effective $\bR$-divisor $E$ such that $A+E\sim_\bR \KFB$, and every component of $\Supp A$ is $\cG$-non-invariant. Set $D:=A+E$.

If $t<1$, set
\[\Gamma:=D^{\ninv}+\frac{1}{1-t}D^{\inv}.\]
If $t=1$, then Step~1 of Proposition~\ref{control singularities} implies that no prime divisor on $Y$ is $\cG$-invariant, and hence we set
\[\Gamma:=D=D^{\ninv}.\]
In either case, we have
\[\Gamma^{\ninv}+(1-t)\Gamma^{\inv}=D.\]

We may take a sufficiently small positive real number $\epsilon$ such that
\[\fB':=(Y,\cG,B_Y+\epsilon\Gamma,t)\]
is klt. Moreover,
\[(B_Y+\epsilon\Gamma)^{\ninv}\geq\epsilon A,\]
and
\[K_{\fB'}=\KFB+\epsilon D\sim_\bR(1+\epsilon)\KFB.\]
In particular, $K_{\fB'}$ is nef. Therefore, by Theorem~\ref{good minimal models}, $Y$ is the good minimal model of $\fB'$, and hence $K_{\fB'}$ is semi-ample. It follows that $\KFB$ is semi-ample, defining a birational morphism $g:Y\to Z$. Let $\fC$ be the pushforward of $\fB$. Then $\KFB=g^*\KFC$, and $\KFC$ is ample. Since $\fB$ is klt and $g$ is $\KFB$-trivial, $\fC$ is klt. The uniqueness of $Z$ follows from the uniqueness of the ample model.
\end{proof}
\end{cor}

\section{Finiteness of models}\label{Section:Finiteness}

In this section, we prove the finiteness of good minimal models and ample models when the adjoint foliated structure varies in a polytope. This is an adjoint foliated analogue of \cite[Corollary~1.1.5]{BCHM10}.

We first prove an elementary topological lemma which will be used to reduce the proof to the boundary of the polytope.

\begin{lem}\label{topology}
Assume that 
\begin{enumerate}
    \item $P\subseteq \bR^n$ is a polytope of dimension $n$,
    \item $Q\subseteq \bR^n$ is a bounded closed set,
    \item $\Gamma:P\to Q$ is a homeomorphism, and
    \item $p_1\in P$, $p_2\in \Int P$ such that $p_2\neq p_1$.
\end{enumerate}
Then there exist a point $p_3\in \partial P$ and a real number $\epsilon\in (0,1)$ such that
\[\Gamma(p_2)=\epsilon\Gamma(p_1)+(1-\epsilon)\Gamma(p_3).\]
Furthermore, for any sufficiently small $\epsilon'>0$, there exists a point $p_4\in \Int P$ such that 
\[\Gamma(p_2)=\epsilon'\Gamma(p_1)+(1-\epsilon')\Gamma(p_4).\]
\end{lem}
\begin{proof}
Consider the ray \[R:=\{\Gamma(p_2)+t(\Gamma(p_2)-\Gamma(p_1))\mid t\geq 0\}.\]
Since $p_2\in \Int P$ and $\Gamma$ is a homeomorphism, it follows that $\Gamma(p_2)\in \Int Q$. As $Q$ is bounded, the ray $R$ cannot be contained entirely in $Q$. Hence $R\cap \partial Q\neq \emptyset$. Choose a point $q_3\in R\cap \partial Q$, and take $p_3=\Gamma^{-1}(q_3)$. Then there exists a real number $\epsilon\in (0,1)$ such that \[\Gamma(p_2)=\epsilon\Gamma(p_1)+(1-\epsilon)\Gamma(p_3).\]
Since homeomorphisms preserve boundaries, it follows that $p_3\in \partial P$.

For furthermore part, since $\Gamma(p_2)\in \Int Q$ and $\epsilon'$ is sufficiently small, there exists a point $q_4\in R\cap \Int Q$ such that 
\[\Gamma(p_2)=\epsilon'\Gamma(p_1)+(1-\epsilon')q_4.\]
Therefore, we can take $p_4=\Gamma^{-1}(q_4)\in \Int P$ and this completes the proof.
\end{proof}

The following elementary observation allows us to deduce global ampleness from relative ampleness.

\begin{lem}\label{from relative ample to ample}
    Let $f:X\to Z$ be a contraction between projective normal varieties. Let $H$ be an ample $\bR$-divisor on $Z$, and $N$ be a nef $\bR$-divisor on $X$ which is ample over $Z$. Then $N+af^*H$ is ample for any positive real number $a$.
\end{lem}
\begin{proof}
    Since $H$ is ample and $N$ is ample over $Z$, there exists a positive real number $b>a$ such that $N+bf^*H$ is ample. Since $N$ is nef, it follows that
    \[N+af^*H=(1-\frac{a}{b})N+\frac{a}{b}(N+bf^*H)\] is ample. 
\end{proof}

We then introduce the parameter spaces in which the adjoint foliated structures will vary.

\begin{definition}
Assume that 
\begin{itemize}
    \item $X$ is a projective normal surface,
    \item $\cF$ is a rank one foliation on $X$,
    \item $V$ is a finite dimensional affine subspace of $\WDiv_\bR(X)\times \bR$, where $\WDiv_\bR(X)$ is the real vector space of Weil divisors on $X$, and
    \item $A\geq 0$ is an $\bR$-divisor on $X$.
\end{itemize}
  Then we define
\begin{enumerate}
    \item $V_A=\{(A+\Delta,t)|(\Delta,t)\in V\}$,
    \item $\cL_A(V)\subseteq V_A$ as the subset such that $\Delta\geq 0$, $t\in [0,1]$, and $\fA:=(X,\cF,A+\Delta,t)$ is log canonical, and
    \item $\cE_A(V)\subseteq \cL_A(V)$ as the subset such that $\KFA$ is pseudo-effective.
\end{enumerate}
\end{definition}

We now prove the following finiteness theorem.

\begin{thm}\label{finiteness of models}
Let $0<\delta<1$ be a real number. Assume that
\begin{enumerate}
\item $X$ is a projective normal surface,
\item $\cF$ is a rank one foliation on $X$,
\item $V$ is a finite dimensional affine subspace of $\WDiv_\bR(X)\times\bR$,
\item $A\geq0$ is a general ample $\cF$-non-invariant $\bR$-divisor, and
\item $\cC$ is a polytope in the interior of $\cL_A(V)$ such that if $(B,t)\in\cC$, then $\delta\leq t\leq1-\delta$ and $\fA(B,t):=(X,\cF,B,t)$ is klt.
\end{enumerate}
Then there exist finitely many diagrams
\[X\xrightarrow{f_i}Y_i\xrightarrow{g_i}Z_i,\qquad 1\leq i\leq k,\]
where $f_i$ is a birational morphism and $g_i$ is a contraction, satisfying the following:

If $(B,t)\in\cC\cap\cE_A(V)$, then there exists an index $1\leq i\leq k$ such that $f_i:X\to Y_i$ is the good minimal model of $\fA(B,t)$ and $g_i\circ f_i:X\to Z_i$ is its ample model.

\begin{proof}
The proof adapts the argument of \cite[Theorem~5.1]{MZ23MMPforlocallystablefamilies} while keeping track of the good minimal models. Possibly replacing $V_A$ by the affine span of $\cC$, we may assume that $\cC$ spans $V_A$. We proceed by induction on $\dim\cC$. If $\dim\cC=0$, the result follows from Theorem~\ref{good minimal models}. Thus, we may assume that $\dim\cC>0$.

\begin{enumerate}[label=\textsl{Step} \arabic{enumi}., wide=13pt, itemsep=13pt]
\item In this step, we reduce the proof to the case where $K_{\fA(B_0,t_0)}$ is semiample for some $(B_0,t_0)\in\cC\cap\cE_A(V)$.

Since $\cC\cap\cE_A(V)$ is compact, it suffices to prove the result locally around every point $(B_0,t_0)\in\cC\cap\cE_A(V)$. By Theorem~\ref{good minimal models}, $\fA(B_0,t_0)$ has a good minimal model $f:X\to Y$. For any $(B,t)$, denote by $B_Y:=f_*B$ and \[\fB(B_Y,t):=(Y,f_*\cF,B_Y,t)\] the induced adjoint foliated surface on $Y$.

Let $\cP$ be a sufficiently small polytope neighborhood of $(B_0,t_0)$ contained in the interior of $\cL_A(V)$. Since $\fB(B_{0,Y},t_0)$ is klt and $f$ is $K_{\fA(B_0,t_0)}$-negative, after shrinking $\cP$, we may assume that $\fB(B_Y,t)$ is klt and \[a(E,\fB(B_Y,t))>a(E,\fA(B,t))\] for every $(B,t)\in\cP$ and every $f$-exceptional prime divisor $E$. Hence $f$ is $K_{\fA(B,t)}$-negative for every $(B,t)\in\cP$, and $\fA(B,t)$ and $\fB(B_Y,t)$ have the same good minimal model and the same ample model.

Replacing $X$, $\cF$, and $\cP$ by $Y$, $f_*\cF$, and $f_*\cP$, respectively, we may assume that $K_{\fA(B_0,t_0)}$ is semiample.

\item In this step, we reduce the proof to the boundary of a polytope while preserving both the good minimal model and the ample model.

Let $\cC_0'$ and $\cC_0$ be polytope neighborhoods of $(B_0,t_0)$ such that \[(B_0,t_0)\in\Int\cC_0'\subseteq\cC_0'\subseteq\Int\cC_0\subseteq\cC_0\subseteq\Int\cL_A(V),\] and such that $\fA(B,t)$ is klt and $\delta\leq t\leq1-\delta$ for every $(B,t)\in\cC_0$. Fix a point $(B_1,t_1)\in\Int\cC_0'\cap\cE_A(V)$ distinct from $(B_0,t_0)$.

For any $(B,t)\in\cC_0$, set \[\Gamma(B,t):=\frac{t}{1-t}\KCF+\frac{1}{1-t}\nB+\iB.\] Then \[K_{\fA(B,t)}=(1-t)(K_X+\Gamma(B,t)).\] Thus, $K_{\fA(B,t)}$ and $K_X+\Gamma(B,t)$ have the same good minimal model and the same ample model.

Let $\mu:X\to Z$ be the contraction induced by the semiample divisor $K_{\fA(B_0,t_0)}$. Then \[K_X+\Gamma(B_0,t_0)\sim_\bR\mu^*H_Z,\] where $H_Z$ is an ample $\bR$-divisor on $Z$. By the proof of \cite[Proposition~3.2(5)]{Bir11OnExistenceLogMinimalModelsII} and Theorem~\ref{conethm}, after shrinking $\cC_0$ and $\cC_0'$, we may assume that every $K_{\fA(B,t)}$-MMP, for $(B,t)\in\cC_0$, is $K_{\fA(B_0,t_0)}$-trivial.

Regarding the coefficient of $\KCF$ as an additional coordinate, the map $\Gamma$ is identified with \[(B,t)\longmapsto\left(\frac{1}{1-t}\nB+\iB,\frac{t}{1-t}\right),\] and restricts to a homeomorphism on each of $\cC_0$ and $\cC_0'$. Applying Lemma~\ref{topology} to $\cC_0$ and $\cC_0'$ along the same ray, we obtain points $(B_2,t_2)\in\partial\cC_0$ and $(B_3,t_3)\in\partial\cC_0'$, and real numbers $0<\epsilon'<\epsilon<1$, such that \[\Gamma(B_1,t_1)=\epsilon\Gamma(B_0,t_0)+(1-\epsilon)\Gamma(B_2,t_2)\] and \[\Gamma(B_1,t_1)=\epsilon'\Gamma(B_0,t_0)+(1-\epsilon')\Gamma(B_3,t_3).\] Hence \[\Gamma(B_3,t_3)=\frac{\epsilon-\epsilon'}{1-\epsilon'}\Gamma(B_0,t_0)+\frac{1-\epsilon}{1-\epsilon'}\Gamma(B_2,t_2).\]

For $0\leq i\leq3$, set \[D_i:=K_X+\Gamma(B_i,t_i)=\frac{1}{1-t_i}K_{\fA(B_i,t_i)}.\] Then \[D_1=\epsilon D_0+(1-\epsilon)D_2\] and \[D_3=\frac{\epsilon-\epsilon'}{1-\epsilon'}D_0+\frac{1-\epsilon}{1-\epsilon'}D_2.\] Since $D_0\sim_{\bR,Z}0$, we have \[D_1\sim_{\bR,Z}(1-\epsilon)D_2.\] As $D_1$ is pseudo-effective, $D_2$ is pseudo-effective over $Z$.

Run a $D_2$-MMP \[g:X\to Y.\] Every step of this MMP is $D_0$-trivial and hence is over $Z$. Since $D_2$ is pseudo-effective over $Z$, this MMP cannot terminate with a Mori fiber space. Therefore, it terminates with $D_{2,Y}:=g_*D_2$ nef. In particular, $D_2$ is pseudo-effective. By Theorem~\ref{good minimal models}, $g:X\to Y$ is the good minimal model of $\fA(B_2,t_2)$ and $D_{2,Y}$ is semiample.

Denote the pushforward of $D_i$ by $D_{i,Y}$. Since every step of $g$ is $D_0$-trivial, it is both $D_1$-negative and $D_3$-negative. Let $h:Y\to W$ be the contraction over $Z$ induced by $D_{2,Y}$, and let $\nu:W\to Z$ be the induced morphism. Then \[D_{2,Y}\sim_\bR h^*H_W,\] where $H_W$ is ample over $Z$ and globally nef, and \[D_{0,Y}\sim_\bR h^*\nu^*H_Z.\] Consequently, it follows from Lemma~\ref{from relative ample to ample} that \[D_{1,Y}\sim_\bR h^*((1-\epsilon)H_W+\epsilon\nu^*H_Z)\] is the pullback of an ample $\bR$-divisor on $W$. Similarly, \[D_{3,Y}\sim_\bR h^*\left(\frac{1-\epsilon}{1-\epsilon'}H_W+\frac{\epsilon-\epsilon'}{1-\epsilon'}\nu^*H_Z\right)\] is also the pullback of an ample $\bR$-divisor on $W$.

It follows that $g:X\to Y$ is the good minimal model of both $\fA(B_1,t_1)$ and $\fA(B_3,t_3)$, while $h\circ g:X\to W$ is their common ample model. In particular, $(B_3,t_3)\in\cE_A(V)$.

Thus, every point of $\Int\cC_0'\cap\cE_A(V)$ distinct from $(B_0,t_0)$ has the same good minimal model and the same ample model as a point of $\partial\cC_0'\cap\cE_A(V)$. Together with the diagram associated with $\fA(B_0,t_0)$, the induction hypothesis applied to the finitely many faces of $\partial\cC_0'$ proves the desired local finiteness.
\end{enumerate}
\end{proof}
\end{thm}

The following consequence keeps track of the models and their compatibility along the faces of the polyhedral decomposition.

\begin{cor}\label{finite model decomposition}
In the setting of Theorem~\ref{finiteness of models}, for $(B,t)\in\cC$, set \[\Gamma(B,t):=\frac{t}{1-t}\KCF+\frac{1}{1-t}\nB+\iB.\] Suppose that $\cP:=\Gamma(\cC)$ is a polytope, and set \[\cP_{\mathrm{pe}}:=\Gamma(\cC\cap\cE_A(V)).\] Then there exist a finite polyhedral decomposition \[\cP_{\mathrm{pe}}=\bigcup_{i=1}^p\cP_i\] and finitely many diagrams \[X\xrightarrow{f_i}Y_i\xrightarrow{g_i}Z_i,\qquad 1\leq i\leq p,\] such that, for every $\Phi=\Gamma(B,t)\in\cP_i^\circ$, the morphism $f_i$ is the good minimal model of $\fA(B,t)$ and $g_i\circ f_i$ is its ample model.

Moreover, the decomposition may be chosen to be compatible along faces. More precisely, if $\cP_j$ is a face of $\cP_i$, then there exists a contraction $\alpha_{ij}:Z_i\to Z_j$ such that \[\alpha_{ij}\circ g_i\circ f_i=g_j\circ f_j.\] 

For every $\Phi\in\cP_j^\circ$, the divisor induced by $K_X+\Phi$ on $Z_i$ is the pullback under $\alpha_{ij}$ of an ample $\bR$-divisor on $Z_j$. In particular, it is semiample, and its semiample contraction is $\alpha_{ij}$.
\end{cor}

\begin{proof}
We keep track of the inductive construction in the proof of Theorem~\ref{finiteness of models}. At each induction step, set $\Phi_0:=\Gamma(B_0,t_0)$, subdivide the boundary into finitely many polytopes on whose relative interiors the good minimal model and the ample model are fixed, and take cones with vertex $\Phi_0$. By Step~2, the relative interior of each such cone has the same good minimal model and the same ample model as the relative interior of its base. This gives the required finite polyhedral decomposition.

It remains to verify the compatibility along faces. We prove this by induction on $\dim\cP_i+\dim\cP_j$. The case $\cP_i=\cP_j$ is immediate, so we may assume that $\cP_j$ is a proper face of $\cP_i$.

By the construction in Step~2, we may write $\cP_i=\operatorname{Conv}(\Phi_0,\cQ_i)$, where $\cQ_i$ is a polytope in the boundary decomposition. Moreover, $\cP_i^\circ$ and $\cQ_i^\circ$ have the same good minimal model and the same ample model, so we use the same diagram for them. If $\cP_j$ does not contain $\Phi_0$, then $\cP_j$ is a face of $\cQ_i$, and the conclusion follows from the induction hypothesis. Hence we may assume that $\cP_j$ contains $\Phi_0$. For $\Phi\in\cP_i$, write $L_i(\Phi)$ for the divisor induced by $K_X+\Phi$ on $Z_i$. We distinguish the following two cases.

\emph{Case 1.} Assume that $\cP_j=\{\Phi_0\}$. The construction in Step~2 gives a contraction $\alpha_{ij}:Z_i\to Z_j$ such that \[\alpha_{ij}\circ g_i\circ f_i=g_j\circ f_j\] and \[L_i(\Phi_0)\sim_\bR\alpha_{ij}^*H_j\] for some ample $\bR$-divisor $H_j$ on $Z_j$. Hence $L_i(\Phi_0)$ is semiample, and its semiample contraction is $\alpha_{ij}$.

\emph{Case 2.} Assume that $\cP_j\neq\{\Phi_0\}$, and set $\cQ_j:=\cP_j\cap\cQ_i$. Then $\cQ_j$ is a face of $\cQ_i$ and \[\cP_j=\operatorname{Conv}(\Phi_0,\cQ_j).\] By Step~2, $\cP_j^\circ$ and $\cQ_j^\circ$ have the same good minimal model and the same ample model, so we use the same diagram for them. The induction hypothesis applied to the face inclusion $\cQ_j\subseteq\cQ_i$ gives a contraction $\alpha_{ij}:Z_i\to Z_j$ such that \[\alpha_{ij}\circ g_i\circ f_i=g_j\circ f_j.\] Moreover, for every $\Psi\in\cQ_j^\circ$, there exists an ample $\bR$-divisor $H_j(\Psi)$ on $Z_j$ such that \[L_i(\Psi)\sim_\bR\alpha_{ij}^*H_j(\Psi).\]

Let $Z_0$ be the ample model corresponding to $\{\Phi_0\}$. Applying Case~1 to the inclusions $\{\Phi_0\}\subseteq\cP_i$ and $\{\Phi_0\}\subseteq\cP_j$, we obtain contractions $\alpha_{i0}:Z_i\to Z_0$ and $\alpha_{j0}:Z_j\to Z_0$. Since $g_i\circ f_i$ is surjective, the corresponding commutative relations give \[\alpha_{i0}=\alpha_{j0}\circ\alpha_{ij}.\] If $H_0$ denotes the ample $\bR$-divisor on $Z_0$ induced by $K_X+\Phi_0$, then \[L_i(\Phi_0)\sim_\bR\alpha_{i0}^*H_0=\alpha_{ij}^*\alpha_{j0}^*H_0.\]

Fix $\Phi\in\cP_j^\circ$. We may write $\Phi=\lambda\Phi_0+(1-\lambda)\Psi$ for some $0<\lambda<1$ and $\Psi\in\cQ_j^\circ$. Then we have \[L_i(\Phi)\sim_\bR \lambda L_i(\Phi_0)+(1-\lambda)L_i(\Psi)\sim_\bR\alpha_{ij}^*\big(\lambda\alpha_{j0}^*H_0+(1-\lambda)H_j(\Psi)\big).\] Since $\alpha_{j0}^*H_0$ is nef and $H_j(\Psi)$ is ample, the divisor in parentheses is ample. Thus, $L_i(\Phi)$ is the pullback under $\alpha_{ij}$ of an ample $\bR$-divisor on $Z_j$. In particular, it is semiample, and its semiample contraction is $\alpha_{ij}$.
\end{proof}

The following lemma transfers bigness from a closed fiber to the generic fiber uniformly over the parameter polytope, and will be used later in the proof of boundedness.

\begin{lem}\label{bigness on the generic fiber}
Let $\pi:\cX\to Z$ be a projective family of smooth surfaces over an integral variety $Z$ with generic point $\eta$, and let $\cC$ be a convex polytope of $\bR$-Cartier divisors on $\cX$. For any $L\in\cC$ and any point $z\in Z$, write $L_z:=L|_{\cX_z}$. Assume that
\begin{enumerate}
\item there exists $L_0\in\cC$ such that $L_{0,\eta}$ is big, and
\item there exist finitely many diagrams \[\cX_\eta\xrightarrow{f_i}Y_i\xrightarrow{g_i}T_i,\qquad 1\leq i\leq k,\] where $f_i$ is a birational morphism, $Y_i$ is $\bQ$-factorial, and $g_i$ is a contraction, such that for every $L\in\cC$ with $L_\eta$ pseudo-effective, one of these diagrams gives the good minimal model and the ample model of $L_\eta$.
\end{enumerate}
Then, after replacing $Z$ by a finite cover and shrinking it, the following holds: for every $L\in\cC$, if $L_z$ is big for some closed point $z\in Z$, then $L_\eta$ is big.

\begin{proof}
After shrinking $Z$, we may assume that for any $L\in \cC$, every irreducible component of $\Supp(L)$ dominates $Z$.

For every $i$ with $\dim T_i\leq1$, choose a covering family of curves on $Y_i$ contracted by $g_i$, and denote a general member by $C_i$. If $\dim T_i=1$, we may take the fibers of $g_i$, while if $\dim T_i=0$, we may take general members of the linear system of a very ample divisor on $Y_i$. After replacing $Z$ by a finite cover and shrinking it, we may extend all the diagrams and these covering families. In particular, we may assume that $f_{i,z}$ is birational and $C_{i,z}$ is movable for every closed point $z\in Z$.

Fix $L_1\in\cC$ and suppose that $L_{1,z}$ is big for some closed point $z\in Z$, but $L_{1,\eta}$ is not big. For $0\leq s\leq1$, set \[L_s:=sL_1+(1-s)L_0.\] Since $\cC$ is convex and $L_{0,\eta}$ is big, there exists a largest real number $0<\lambda\leq1$ such that $L_{\lambda,\eta}$ is pseudo-effective. Then $L_{\lambda,\eta}$ is pseudo-effective but not big.

Choose a diagram \[\cX_\eta\xrightarrow{f_i}Y_i\xrightarrow{g_i}T_i\] giving the good minimal model and the ample model of $L_{\lambda,\eta}$. Since $L_{\lambda,\eta}$ is not big, we have $\dim T_i\leq1$. Denote \[G_s:=(f_i)_*L_{s,\eta}.\] Since $g_i$ is the ample model of $L_{\lambda,\eta}$, we have 
$G_{\lambda}\cdot C_i=0$. On the other hand, $G_0$ is big and $C_i$ is movable, and hence $G_0\cdot C_i>0$. Since \[G_\lambda=\lambda G_1+(1-\lambda)G_0,\] it follows that \[G_1\cdot C_i\leq0.\]
After specialization to $z$, we obtain \[(f_{i,z})_*L_{1,z}\cdot C_{i,z}\leq0.\]
However, $(f_{i,z})_*L_{1,z}$ is big because $L_{1,z}$ is big and $f_{i,z}$ is birational. Since $C_{i,z}$ is movable, this gives \[(f_{i,z})_*L_{1,z}\cdot C_{i,z}>0,\] a contradiction. Therefore, $L_{1,\eta}$ is big.
\end{proof}
\end{lem}

\section{Boundedness of polarized foliated surfaces}\label{Section:Boundedness}

\subsection{Effective birationality}
We first establish an effective birationality result for polarized adjoint foliated surfaces, which will later be used to construct log bounded birational models. The main input is Birkar's effective birationality theorem for adjoint linear series \cite[Corollary 1.2]{Bir23GeometryPolarisedVarieties}.

\begin{prop}\label{effective birationality}
Let $\epsilon$ be a positive real number, and $I\subseteq \bQ\cap(0,1)$ be a finite set. Then there exist positive integers $m,n$ depending only on $\epsilon,I$ satisfying the following: Assume that 
\begin{enumerate}
    \item $(X,\cF,B,t)$ is an $\epsilon$-lc adjoint foliated surface with $t\in I$,
    \item the non-zero coefficients of $\nB$ are in $I$,
    \item $\KCF+\nB$ is pseudo-effective, and
    \item $N$ is a nef and big integral divisor on $X$.
\end{enumerate}
Then $|m(t(\KCF+\nB)+(1-t)K_X)+nN+L|$ defines a birational map for any pseudo-effective integral divisor $L$.

\begin{proof}
By Proposition~\ref{control singularities}, $X$ is $\epsilon$-lc.
Applying \cite[Corollary~1.2]{Bir23GeometryPolarisedVarieties},
there exist positive integers $m_0,l_0$, depending only on $\epsilon$, such that 
\[|m'K_X+l'N+P|\] defines a birational map for any positive
integers $m'\geq m_0$ and $l'\geq l_0m'$, and any pseudo-effective integral divisor $P$.

Since $I\subseteq\bQ\cap(0,1)$ is finite and the non-zero coefficients
of $\nB$ belong to $I$, we may choose a sufficiently divisible positive
integer $m$, depending only on $\epsilon,I$, such that, for every
$t\in I$, the divisor $mt(\KCF+\nB)$ is integral,
$m(1-t)\in\bN^{>0}$, and $m(1-t)\geq m_0$.
Set $n:=l_0m$.
Then $n\geq l_0m(1-t)$.
Moreover, \[P:=mt(\KCF+\nB)+L\] is a pseudo-effective integral divisor. It follows that
\[|m(t(\KCF+\nB)+(1-t)K_X)+nN+L|\]
defines a birational map.
\end{proof}
\end{prop}

As a consequence, we obtain effective birationality for adjoint foliated surfaces of general type, extending \cite[Theorem 1.4]{SS23EffectiveGenerationFoliatedSurfaces}.

\begin{cor}\label{effective birationality of general type}
    Let $\epsilon$ be a positive real number, and $I\subseteq \bQ\cap(0,1)$ be a finite set. Then there exists a positive integer $m$ depending only on $\epsilon,I$ satisfying the following: Assume that 
\begin{enumerate}
    \item $\fA:=(X,\cF,B,t)$ is an $\epsilon$-lc adjoint foliated surface with $t\in I$,
    \item the non-zero coefficients of $B$ are in $I$, and
    \item $\KFA$ is big.
\end{enumerate}
Then $|m\KFA|$ defines a birational map.

\begin{proof}
Let $f:X\to Y$ be the minimal model of $\KFA$ constructed by
Theorem~\ref{adjoint MMP}, and let
$\fB:=(Y,\cG,B_Y,t)$ be the pushforward of $\fA$.
Then $\fB$ is $\epsilon$-lc by the negativity lemma, and $\KFB$ is
nef and big.

Since the non-zero coefficients of $B$ and $t$ belong to the finite
set $I$, we may choose a positive integer $p$, depending only on $I$,
such that $p\KFA$, $p\KFB$, and $p(1-t)B_Y$ are integral.
By \cite[Lemma~5.3]{LX25NonalgebraictyNonabundant},
$K_{\cG}+B_Y^{\ninv}$ is pseudo-effective.

By Proposition~\ref{effective birationality}, there exist positive integers $m_0,n_0$,
depending only on $\epsilon,I$, such that
\[\big|m_0\big(t(K_{\cG}+B_Y^{\ninv})+(1-t)K_Y\big)+n_0p\KFB\big|\]
defines a birational map. Then 
\[|m_0p(t(\KCG+\nB_Y)+(1-t)(K_Y+B_Y))+n_0p^2\KFB|,\]
and hence \[|(m_0p+n_0p^2)\KFB|\]
defines a birational map. Since $\KFA\geq f^*\KFB$, it follows that $|m\KFA|$ defines a birational map, where $m=m_0p+n_0p^2$.
\end{proof}
\end{cor}

\subsection{Deriving boundedness from birational boundedness}

In this subsection, we prove the following theorem, which is the adjoint foliated surface version of \cite[Theorem 1.6]{HMX14ACCLCT}.

\begin{thm}\label{Deriving boundedness from birational boundedness}
    Let $\epsilon$ and $\delta$ be positive real numbers. Consider the set $\cP$ of $\fA=(X,\cF,B,t)$ such that 
    \begin{enumerate}
        \item $\fA$ is an $\epsilon$-lc adjoint foliated surface, 
        \item the non-zero coefficients of $B$ are $\geq \delta$,
        \item $\delta\leq t\leq 1-\delta$, and
        \item $\KFA$ is ample.
    \end{enumerate}
    If $\cP$ is log birationally bounded, then $\cP$ is a log bounded family.
\end{thm}

We begin with some preparations. The following resolution lemma will be applied to the generic fiber of the bounded family. It provides a foliation-adapted log smooth model on which the components of the boundary appearing in the pullback
formula are pairwise disjoint. After spreading out this model, the separation of these components will allow us to obtain terminal adjoint foliated structures on every fiber.

\begin{lem}\label{components not intersect}
Let $\fA:=(X,\cF,B,t)$ be a klt adjoint foliated surface with $t<1$. Then there exists a resolution $\pi:Y\to X$ satisfying the following:

Assume that $\cG$ is the pullback of $\cF$, and write $\fB:=(Y,\cG,D,t)$ to be an adjoint foliated surface such that
\[\KFB=\pi^*\KFA+F,\]
where $D,F$ are two unique effective $\bR$-divisors with no common components. Then $(Y,D)$ is $\cG$-adapted log smooth and the irreducible components of $\Supp D$ are pairwise disjoint.

\begin{proof}
Let $\mu:X'\to X$ be a log resolution of $(X,B)$ and let $\cF'$ be the pullback of $\cF$ such that
\begin{itemize}
\item $\cF'$ has reduced singularities, and
\item $(X',\mu_*^{-1}B+\Exc(\mu))$ is $\cF'$-adapted log smooth.
\end{itemize}
Take $\fA':=(X',\cF',B',t)$ to be the adjoint foliated surface such that
\[K_{\fA'}=\mu^*\KFA+N,\]
where $B',N$ are effective $\bR$-divisors with no common components. Since $\fA$ is klt, every coefficient of $B'$ is strictly less than one. In particular, $(X',B')$ is klt. Since $(X',B'^{\ninv})$ is $\cF'$-adapted log smooth and $\cF'$ has reduced singularities, it follows that $(\cF',B'^{\ninv})$ is canonical.

By \cite[Proposition~2.36]{KM98BirationalGeometry}, there exists a higher resolution $\nu:Y\to X'$ such that, writing
\[K_Y+D'=\nu^*(K_{X'}+B')+F',\]
where $D',F'$ are effective $\bR$-divisors with no common components, the irreducible components of $\Supp D'$ are pairwise disjoint. By the construction in \cite[Proposition~2.36]{KM98BirationalGeometry}, the morphism $\nu$ may be chosen as a sequence of blow-ups at intersection points of irreducible components of the boundaries appearing in the pullback formulas. Since $\cF'$ has reduced singularities and $(X',B')$ is $\cF'$-adapted log smooth, each of these blow-ups preserves foliation-adapted log smoothness. Thus, $(Y,D')$ is $\cG$-adapted log smooth, where $\cG$ is the pullback of $\cF'$.

Take $\fB:=(Y,\cG,D,t)$ to be the adjoint foliated surface such that
\[\KFB=\nu^*\mu^*\KFA+F,\]
where $D,F$ are effective $\bR$-divisors with no common components. We claim that
\[\Supp D\subseteq\Supp D'.\]

Let $E$ be an irreducible component of $\Supp D$. If $\nu_*E$ is a divisor on $X'$, then $E$ is also a component of $\Supp D'$ because $\nu_*D=\nu_*D'=B'$. Thus, we may assume that $E$ is $\nu$-exceptional. Since $(\cF',B'^{\ninv})$ is canonical, we have
\[a(E,\cF',B'^{\ninv})\geq 0.\]
By the construction of $\fA'$ and $D$, we have
\[a(E,\fA')\leq a(E,\fA)<0.\]
Since $t<1$, combining these two inequalities gives
\[a(E,X',B')<0.\]
Therefore, $E$ is an irreducible component of $\Supp D'$. This proves that $\Supp D\subseteq\Supp D'$.

Since $(Y,D')$ is $\cG$-adapted log smooth and the irreducible components of $\Supp D'$ are pairwise disjoint, the same holds for $(Y,D)$. Therefore, $\pi:=\mu\circ\nu$ is the desired resolution.
\end{proof}
\end{lem}

The following proposition allows us to preserve the log
boundedness of adjoint foliated surfaces after extracting divisors with non-positive discrepancies.

\begin{prop}\label{extract divisor in family}
Let $\epsilon,\delta$ be positive real numbers. Let $\cP$ be a log bounded set of $\epsilon$-lc adjoint foliated surfaces $(X,\cF,B,t)$ with $t\leq 1-\delta$. 

Assume that $\cQ$ is a set of adjoint foliated surfaces such that for every $\fA'=(X',\cF',B',t)\in \cQ$, there exist an adjoint foliated surface $\fA=(X,\cF,B,t)\in\cP$ and a birational morphism $g: X'\to X$ such that $\cF'$ is the pullback of $\cF$ and \[K_{\fA'}=g^*\KFA.\]
Then $\cQ$ is log bounded. 

\begin{proof}
\begin{enumerate} [label=\textsl{Step} \arabic{enumi}., wide=13pt, itemsep=13pt]
\item In this step, we construct the family $(\cY',\cG',\cE')\to Z$ from the log bounded set $\cP$ with the desired properties.

By assumption, there exist a projective morphism $\cY\to Z$, where $Z$ is of finite type, and a foliated couple $(\cG,\cE)$ on $\cY$ such that for every $\fA=(X,\cF,B,t)\in\cP$, there exists a closed point $z\in Z$ such that
\begin{itemize}
\item there is an isomorphism $f:\cY_z\to X$,
\item $f_*\cG_z\simeq\cF$, and
\item $\cE_z$ coincides with the support of $B$.
\end{itemize}
Let $Z_0\subseteq Z$ denote the set of such points. By Noetherian induction, after stratifying $Z$, taking a suitable foliated log resolution over the generic point of each stratum, and spreading it out, we may assume that there exists a birational morphism $\pi:(\cY',\cE')\to(\cY,\cE)$ over $Z$ such that
\begin{itemize}
\item $Z$ is integral,
\item $Z_0$ is Zariski dense in $Z$,
\item $\cE'$ coincides with the sum of the strict transform of $\cE$ and all $\pi$-exceptional divisors,
\item if $\cG'$ is the pullback of $\cG$ and $\eta$ is the generic point of $Z$, then $\cG'_\eta$ has canonical singularities, and
\item $(\cY'_{z'},\cE'_{z'})$ is $\cG'_{z'}$-adapted log smooth for every $z'\in Z$.
\end{itemize}

\item In this step, we construct a resolution $\mu:\cV\to\cY'$ such that the induced map $\tau:\cV_z\dashrightarrow X'$ is a birational morphism.

For $0\leq s\leq1-\delta$, set $\fB(s):=(\cY',\cG',(1-\epsilon)\cE',s)$. Let $\cY'_\eta$ be the generic fiber of $\cY'\to Z$, and denote by $\fB_\eta(s)$ the restriction of $\fB(s)$ to $\cY'_\eta$. Since $\cG'_\eta$ has canonical singularities, it follows from Lemma~\ref{adjoint terminal} that $\fB_\eta(0)$ and $\fB_\eta(1-\delta)$ are klt.

Applying the construction in Lemma~\ref{components not intersect} simultaneously to these two adjoint foliated surfaces, after replacing $Z$ by a finite cover and shrinking it, we obtain a resolution $\mu:\cV\to\cY'$. Let $\cH$ be the pullback of $\cG'$, and write $\fC(s):=(\cV,\cH,\cJ(s),s)$, where $\cJ(s)$ and $\cL(s)$ are effective $\bR$-divisors with no common components such that
\[K_{\fC(s)}=\mu^*K_{\fB(s)}+\cL(s).\]
The conclusions of Lemma~\ref{components not intersect} hold for $s=0$ and $s=1-\delta$ on every fiber, and $\cH_\eta$ has canonical singularities.

Since $\fC_\eta(0)$ and $\fC_\eta(1-\delta)$ are klt, all the coefficients of $\cJ_\eta(0)$ and $\cJ_\eta(1-\delta)$ are strictly smaller than one. As these two divisors have only finitely many components, we may choose a real number $\epsilon_1>0$ such that all their coefficients are at most $1-\epsilon_1$. Fix
\[0<\delta_1<\min\left\{\frac14,\frac{\epsilon_1\delta}{\epsilon_1\delta+1-\delta}\right\}.\]
Since $\cH_\eta$ has canonical singularities, it is $\delta_1$-canonical. By the openness of $\delta_1$-canonicity \cite[Lemma~7.3]{PS19EffectiveAlgebraicIntegration}, after shrinking $Z$, we may assume that $\cH_{z'}$ is $\delta_1$-canonical for every $z'\in Z$. Therefore, Lemma~\ref{adjoint terminal} implies that $\fC_{z'}(0)$ and $\fC_{z'}(1-\delta)$ are terminal. Since $a(E,\fC_{z'}(s))$ is concave in $s$ for every exceptional divisor $E$ over $\cV_{z'}$, it follows that $\fC_{z'}(s)$ is terminal for every $z'\in Z$ and every $0\leq s\leq1-\delta$.

We now fix $\fA'=(X',\cF',B',t)\in\cQ$ and the corresponding $\fA=(X,\cF,B,t)\in\cP$ and birational morphism $g:X'\to X$. Let $z\in Z_0$ be the corresponding point and set $h:=f\circ\pi_z:\cY'_z\to X$. Since $\fA$ is $\epsilon$-lc, we have
\[K_{\fB_z(t)}\geq h^*\KFA.\]
Let $E$ be a $g$-exceptional divisor. The equality $K_{\fA'}=g^*\KFA$ and the effectivity of $B'$ imply that $a(E,\fA)\leq0$, and hence $a(E,\fB_z(t))\leq0$. If $E$ is not a divisor on $\cV_z$, then
\[a(E,\fC_z(t))\leq a(E,\fB_z(t))\leq0,\]
contradicting the terminality of $\fC_z(t)$. Thus, every $g$-exceptional divisor is a divisor on $\cV_z$, and therefore the induced map $\tau:\cV_z\dashrightarrow X'$ is a birational morphism.

\item In this step, we use MMP to prove that $\cQ$ is log bounded.

After a further stratification and Noetherian induction, we may assume that the reduced sum of all $\tau$-exceptional divisors on $\cV_z$ is the restriction of a reduced divisor $\cN$ on $\cV$. We have
\[\Supp\cN\subseteq\Supp(\mu_*^{-1}\cE')\cup\Exc(\mu).\]
Since the divisor on the right has only finitely many irreducible components, there are only finitely many possibilities for $\cN$. It therefore suffices to treat each such possibility separately.

Let $\cJ'_z(t)$ and $\cL'_z(t)$ be effective $\bR$-divisors with no common components such that, setting $\fC'_z(t):=(\cV_z,\cH_z,\cJ'_z(t),t)$, we have
\[K_{\fC'_z(t)}=\tau^*K_{\fA'}+\cL'_z(t).\]
After a further stratification, we may extend these divisors and denote the resulting divisors on $\cV$ by $\cJ'(t)$ and $\cL'(t)$. Since $\cL'_z(t)$ is $\tau$-exceptional, it follows that $\Supp\cL'_z(t)\subseteq\Supp\cN_z$.

For a sufficiently small real number $\alpha>0$, set
\[\fD:=(\cV_\eta,\cH_\eta,\cJ'_\eta(t)+\alpha\cN_\eta,t).\]
Then $\fD$ is klt. Set
\[\cD_\eta:=\cL'_\eta(t)+\alpha(\cN_\eta^{\ninv}+(1-t)\cN_\eta^{\inv}).\]
Since $K_{\fA'}=g^*\KFA$, we have
\[\KFD\sim_{\bR,\cY_\eta}\cD_\eta.\]
The divisor $\cD_\eta$ is effective and exceptional over $\cY_\eta$. Since $t\leq1-\delta$, we have $\Supp\cD_\eta=\Supp\cN_\eta$.

Run a $\KFD$-MMP over $\cY_\eta$. It terminates with a relative minimal model. On the resulting model, the pushforward of $\cD_\eta$ is effective, nef, and exceptional over $\cY_\eta$, and hence it is zero by the negativity lemma. Therefore, every irreducible component of $\cN_\eta$ is contracted, and the output is precisely the contraction determined by $\tau$.

After replacing $Z$ by a finite cover, we may extend this sequence of contractions to obtain a birational morphism $\cV\to\widetilde{\cV}$ over $Z$. Let $\widetilde{\cH}$ and $\widetilde{\cJ'}(t)$ be the pushforwards of $\cH$ and $\cJ'(t)$, respectively. Then
\[(\widetilde{\cV}_z,\widetilde{\cH}_z,\widetilde{\cJ'}_z(t),t)\simeq(X',\cF',B',t).\]
Since there are only finitely many strata and finitely many possibilities for $\cN$, it follows that $\cQ$ is log bounded.
\end{enumerate}
\end{proof}
\end{prop}

The following lemma is a consequence of Theorem~\ref{finiteness of models}.

\begin{lem}\label{generic fiber gmm}
Let $0<\lambda_1<\lambda_2<1$ and $0<\mu_1<\mu_2<1$ be real numbers. Let $\cF$ be a rank one lc foliation on a projective surface $X$, and let $\Sigma$ be a reduced divisor on $X$ such that $(X,\Sigma)$ is $\cF$-adapted log smooth. Denote \[\fA(B,t):=(X,\cF,B,t).\] Assume that $K_{\fA(\mu_2\Sigma,\lambda_2)}$ is big. Then there exist finitely many diagrams \[X\xrightarrow{f_i}Y_i\xrightarrow{g_i}Z_i,\qquad 1\leq i\leq k,\] where $f_i$ is a birational morphism and $g_i$ is a contraction, satisfying the following:

If $\mu_1\Sigma\leq B\leq\mu_2\Sigma$, $\lambda_1\leq t\leq\lambda_2$, and $K_{\fA(B,t)}$ is pseudo-effective, then there exists an index $1\leq i\leq k$ such that $f_i:X\to Y_i$ is the good minimal model of $\fA(B,t)$ and $g_i\circ f_i:X\to Z_i$ is its ample model.
\end{lem}

\begin{proof}
Given an effective $\bR$-divisor $B$ supported on $\Sigma$ and a real number $0\leq t<1$, set \[\Gamma(B,t):=\frac{t}{1-t}\KCF+\frac{1}{1-t}\nB+\iB.\] Then \[K_{\fA(B,t)}=(1-t)(K_X+\Gamma(B,t)).\] Therefore, $K_{\fA(B,t)}$ and $K_X+\Gamma(B,t)$ have the same good minimal model and the same ample model whenever they exist.

Since \[K_X+\Gamma(\mu_2\Sigma,\lambda_2)=\frac{1}{1-\lambda_2}K_{\fA(\mu_2\Sigma,\lambda_2)}\] is big, we may choose an effective $\bR$-divisor $J$ such that \[K_X+\Gamma(\mu_2\Sigma,\lambda_2)\sim_\bR J\] and $J\geq A$, where $A$ is a general effective ample $\cF$-non-invariant $\bR$-divisor.

We regard the coefficient of $\KCF$ in $\Gamma(B,t)$ as an additional coordinate. Applying the furthermore part of Lemma~\ref{topology} to the polytope \[\{(B,t)\mid 0\leq B\leq\mu_2\Sigma,\ 0\leq t\leq\lambda_2\},\] we may choose a sufficiently small real number $\epsilon>0$, a real number $0<\lambda_1'<\lambda_1$, and an effective $\bR$-divisor $D\leq\mu_1\Sigma$ with $\Supp D=\Supp\Sigma$ such that \[\Gamma(\mu_1\Sigma,\lambda_1)=\epsilon\Gamma(\mu_2\Sigma,\lambda_2)+(1-\epsilon)\Gamma(D,\lambda_1').\]

By Lemma~\ref{adjoint terminal}, $\fA(\mu_2\Sigma,t)$ is $(1-\lambda_2)(1-\mu_2)$-lc for every $0\leq t\leq\lambda_2$. Thus, by choosing $\epsilon$ sufficiently small in the preceding paragraph, we may furthermore assume that \[\fA\left(\mu_2\Sigma+\frac{\epsilon}{1-\epsilon}J,t\right)\] is klt for every $0\leq t\leq\lambda_2$.

Now fix $\mu_1\Sigma\leq B\leq\mu_2\Sigma$ and $\lambda_1\leq t\leq\lambda_2$. Comparing the coordinates of $\Gamma(B,t)$, we may choose $B'$ and $t'$ satisfying \[D\leq B'\leq\mu_2\Sigma,\qquad \lambda_1'\leq t'\leq\lambda_2,\] such that \[\Gamma(B,t)=\epsilon\Gamma(\mu_2\Sigma,\lambda_2)+(1-\epsilon)\Gamma(B',t').\] It follows that \[K_X+\Gamma(B,t)\sim_\bR(1-\epsilon)\left(K_X+\frac{\epsilon}{1-\epsilon}J+\Gamma(B',t')\right).\]

Set \[B'':=B'+\frac{\epsilon(1-t')}{1-\epsilon}J^{\ninv}+\frac{\epsilon}{1-\epsilon}J^{\inv}.\] Then \[K_X+\frac{\epsilon}{1-\epsilon}J+\Gamma(B',t')=K_X+\Gamma(B'',t')=\frac{1}{1-t'}K_{\fA(B'',t')}.\] Consequently, \[K_{\fA(B,t)}\sim_\bR\frac{(1-t)(1-\epsilon)}{1-t'}K_{\fA(B'',t')}.\] In particular, $K_{\fA(B,t)}$ is pseudo-effective if and only if $K_{\fA(B'',t')}$ is pseudo-effective, and in this case $\fA(B,t)$ and $\fA(B'',t')$ have the same good minimal model and the same ample model. Since \[\frac{\epsilon(1-t')}{1-\epsilon}J^{\ninv}\geq\frac{\epsilon(1-\lambda_2)}{1-\epsilon}A,\] the result follows from Theorem~\ref{finiteness of models}.
\end{proof}

We are now ready to prove Theorem \ref{Deriving boundedness from birational boundedness}. The strategy of the proof is as follows: Starting from a log birationally bounded family, we first pass to a bounded family of models which extracts the exceptional
divisors with non-positive adjoint discrepancies. This allows us to realize the original adjoint foliated surfaces as log canonical models of members of a bounded family. We then apply Lemma~\ref{generic fiber gmm} to the generic fiber of this bounded family and hence prove boundedness.

\begin{proof}[Proof of Theorem~\ref{Deriving boundedness from birational boundedness}]
Possibly decreasing $\epsilon$ and $\delta$, we may assume that $\epsilon,\delta\in\bQ^{>0}$ and $\delta<\min\{\frac{1}{2},1-\epsilon\}$. We use Noetherian induction throughout the proof.

\begin{enumerate}[label=\textsl{Step} \arabic{enumi}., wide=13pt, itemsep=13pt]
\item In this step, we construct a family $(\cY,\cG,\cE)\to Z$ corresponding to the log birationally bounded set $\cP$ and then modify
this family by stratifying $Z$ and taking a suitable log resolution of $(\cY,\cE)$.

By assumption, there exist a projective morphism $\pi:\cY\to Z$, where $Z$ is of finite type, and a foliated couple $(\cG,\cE)$ on $\cY$ such that for every $\fA=(X,\cF,B,t)\in\cP$, there exist a closed point $z\in Z$ and a birational map $f:\cY_z\dashrightarrow X$ satisfying
\begin{itemize}
\item $f_*\cG_z\simeq\cF$, and
\item $\cE_z$ contains the support of $f_*^{-1}B$ and all $f$-exceptional divisors.
\end{itemize}
Let $Z_0\subseteq Z$ be the set of such points. By stratifying $Z$, replacing it by a finite cover, and taking suitable log resolutions, we may assume that
\begin{itemize}
\item $Z$ is integral and $Z_0$ is Zariski dense in $Z$,
\item $(\cY_z,\cE_z)$ is $\cG_z$-adapted log smooth for every $z\in Z$,
\item $\cG_z$ is lc for every $z\in Z$,
\item every irreducible component of $\cE$ dominates $Z$ and has geometrically irreducible fibers, and
\item for every $\fA=(X,\cF,B,t)\in\cP$ represented by $z\in Z_0$, the divisor $\cE_z$ is the sum of the support of $f_*^{-1}B$ and all $f$-exceptional divisors.
\end{itemize}
The third property follows from Lemma~\ref{delta-canonical family}, while the last property follows after treating the finitely many reduced subdivisors of $\cE$ separately.

\item In this step, we reduce to the case where $f:\cY_z\dashrightarrow X$ is a birational morphism.

Fix $\fA=(X,\cF,B,t)\in\cP$ and the corresponding birational map $f:\cY_z\dashrightarrow X$. Set \[\fB_z:=(\cY_z,\cG_z,(1-\epsilon)\cE_z,t).\] Let $p:W\to X$ and $q:W\to\cY_z$ be common resolutions. Since $\fA$ is $\epsilon$-lc, we have \[q_*p^*\KFA\leq K_{\fB_z}.\] Since $\KFA$ is ample, the negativity lemma gives \[p^*\KFA\leq q^*K_{\fB_z}.\] Therefore, every exceptional divisor of $f^{-1}$ has non-positive discrepancy with respect to $\fB_z$.

Since $(\cY_z,\cE_z)$ is $\cG_z$-adapted log smooth and $\cG_z$ is lc, Lemma~\ref{adjoint terminal} implies that $\fB_z$ is $\delta\epsilon$-lc. By \cite[Lemma~5.4]{LX25NonalgebraictyNonabundant}, there exists a birational morphism $\rho:Y'_z\to\cY_z$ which extracts exactly the exceptional divisors of $f^{-1}$. Let $\cG'_z$ be the pullback of $\cG_z$, and define $\fB'_z=(Y'_z,\cG'_z,B'_z,t)$ by \[K_{\fB'_z}=\rho^*K_{\fB_z}.\] The boundary $B'_z$ is effective because all the extracted divisors have non-positive discrepancies.

Proposition~\ref{extract divisor in family} shows that all such crepant extractions belong to a log bounded family. Replacing $(\cY,\cG,\cE)\to Z$ by this family, we may assume that $f:\cY_z\to X$ is a birational morphism. Repeating the constructions of Step~1, we may preserve all the properties arranged there.

\item In this step, we prove that the adjoint canonical divisor corresponding to the upper corner of the parameter polytope is big on the generic fiber.

For every $z\in Z_0$, we have \[K_{\fB_z}\geq f^*\KFA,\] and hence $K_{\fB_z}$ is big. By \cite[Lemma~5.3]{LX25NonalgebraictyNonabundant}, the divisor $K_{\cG_z}+(1-\epsilon)\cE_z^{\ninv}$ is pseudo-effective. Set \[\fB'_z:=(\cY_z,\cG_z,(1-\epsilon)\cE_z,1-\delta).\] 
Then $K_{\fB'_z}$ is big.

By Corollary~\ref{effective birationality of general type}, there exists a positive integer $M$, depending only on $\epsilon$ and $\delta$, such that $|MK_{\fB'_z}|$ defines a birational map for every $z\in Z_0$. Set \[\fB':=(\cY,\cG,(1-\epsilon)\cE,1-\delta).\] After a further stratification and shrinking $Z$, we may assume that $MK_{\fB'}$ is Cartier, the sheaf \[\cR:=\pi_*\cO_\cY(MK_{\fB'})\] is locally free, and its formation commutes with base change.

Since $Z_0$ is Zariski dense, we may choose $z\in Z_0$ in this open subset. The rational map 
\[\Psi:\cY\dashrightarrow\bP_Z(\cR)\]
restricts to the birational map defined by $|MK_{\fB'_z}|$ on $\cY_z$. Since birationality is an open condition, it follows that the restriction of $\Psi$ to the generic fiber is birational. Therefore, writing \[\fB'_\eta:=(\cY_\eta,\cG_\eta,(1-\epsilon)\cE_\eta,1-\delta),\] the divisor $K_{\fB'_\eta}$ is big.

\item In this step, we prove that every parameter which occurs on a fiber with big adjoint canonical divisor is also big on the generic fiber.

For $\delta\cE\leq\cD\leq(1-\epsilon)\cE$ and $\delta\leq s\leq1-\delta$, set \[\Gamma(\cD,s):=\frac{s}{1-s}\KCG+\frac{1}{1-s}\cD^{\ninv}+\cD^{\inv}.\] Regarding the coefficient of $\KCG$ as an additional coordinate, the set \[\cQ:=\{\Gamma(\cD,s)\mid\delta\cE\leq\cD\leq(1-\epsilon)\cE,\ \delta\leq s\leq1-\delta\}\] is a convex polytope in $\WDiv_{\bR}(\cY)\times \bR$. By Step~3, \[K_{\cY_\eta}+\Gamma((1-\epsilon)\cE_\eta,1-\delta)=\frac{1}{\delta}K_{\fB'_\eta}\] is big. Applying Lemma~\ref{generic fiber gmm} to $(\cY_\eta,\cG_\eta,\cE_\eta)$ with \[\mu_1=\lambda_1=\delta,\qquad \mu_2=1-\epsilon,\qquad \lambda_2=1-\delta,\] we obtain finitely many diagrams which give the good minimal model and the ample model of every pseudo-effective member of $K_{\cY_\eta}+\cQ_\eta$. Therefore, after replacing $Z$ by a finite cover and shrinking it, Lemma~\ref{bigness on the generic fiber} applied to the polytope $K_{\cY/Z}+\cQ$ shows that if one of its members is big on a closed fiber, then it is big on the generic fiber.

Fix $\fA=(X,\cF,B,t)\in\cP$, represented by a point $z\in Z_0$, and let $F_z$ be the reduced $f$-exceptional divisor. Set \[\cD_z:=f_*^{-1}B+(1-\epsilon)F_z.\] Then \[\delta\cE_z\leq\cD_z\leq(1-\epsilon)\cE_z.\] Since the components of $\cE$ have geometrically irreducible fibers, $\cD_z$ extends to a divisor $\cD$ on $\cY$ satisfying \[\delta\cE\leq\cD\leq(1-\epsilon)\cE.\] Set \[\fC_z:=(\cY_z,\cG_z,\cD_z,t),\qquad \fC_\eta:=(\cY_\eta,\cG_\eta,\cD_\eta,t).\] We have \[K_{\fC_z}=f^*\KFA+R_z\] for some effective $f$-exceptional divisor $R_z$, and hence $K_{\fC_z}$ is big. Thus, \[K_{\cY_z}+\Gamma(\cD_z,t)=\frac{1}{1-t}K_{\fC_z}\] is big. By the preceding paragraph, $K_{\cY_\eta}+\Gamma(\cD_\eta,t)$ is big, and therefore $K_{\fC_\eta}$ is big.

\item In this step, we extend the finitely many ample models and identify their fibers with the original adjoint foliated surfaces.

Applying Corollary~\ref{finite model decomposition} in place of Theorem~\ref{finiteness of models} in the proof of Lemma~\ref{generic fiber gmm}, we obtain a finite polyhedral decomposition of the pseudo-effective locus of $\cQ_\eta$ and finitely many diagrams \[\cY_\eta\xrightarrow{\overline f_i}\overline{\cW}_i\xrightarrow{\overline g_i}\overline{\cV}_i,\qquad 1\leq i\leq p,\] which are compatible along faces. After replacing $Z$ by a finite cover and shrinking it, we may extend these diagrams and all the contractions corresponding to inclusions of faces to obtain \[\cY\xrightarrow{f_i}\cW_i\xrightarrow{g_i}\cV_i.\] 

For each $i$, let $\cP_i$ be the polytope corresponding to the $i$-th diagram, fix a point $\Phi_i\in\cP_i^\circ$, and denote by $L_i(\Phi)$ the divisor induced by $K_{\cY_\eta}+\Phi$ on $\cV_{i,\eta}$. Spreading out the pullback identities at the vertices of the finitely many polytopes and the compatibility identities along their faces, and shrinking $Z$ further, we may assume that all these identities hold on every fiber. Since there are only finitely many $i$, openness of ampleness allows us to assume moreover that $L_{i,z}(\Phi_i)$ is ample for every $i$ and every $z\in Z$.

We claim by induction on $\dim\cP_i$ that $L_{i,z}(\Phi)$ is ample for every $\Phi\in\cP_i^\circ$. This is clear if $\dim\cP_i=0$. Otherwise, for $\Phi\neq\Phi_i$, we may write \[\Phi=\lambda\Phi_i+(1-\lambda)\Psi,\] where $0<\lambda<1$ and $\Psi\in\cP_j^\circ$ for some proper face $\cP_j$ of $\cP_i$. By compatibility along faces, \[L_{i,z}(\Psi)\sim_{\bR}\alpha_{ij,z}^*L_{j,z}(\Psi).\] By induction, $L_{j,z}(\Psi)$ is ample, and hence $L_{i,z}(\Psi)$ is nef. Therefore, \[L_{i,z}(\Phi)\sim_{\bR}\lambda L_{i,z}(\Phi_i)+(1-\lambda)L_{i,z}(\Psi)\] is ample. It follows from the corresponding pullback identity that $g_{i,z}\circ f_{i,z}$ is the ample model of $K_{\cY_z}+\Phi_z$ for every $\Phi\in\cP_i^\circ$ and every $z\in Z$.

Return to $\fA$ and $\cD$ chosen in Step~4, and set \[\Phi:=\Gamma(\cD_\eta,t).\] By Step~4, $K_{\cY_\eta}+\Phi$ is big. Choose a polytope whose relative interior contains $\Phi$, and write \[\varphi_i:=g_i\circ f_i:\cY\to\cV_i.\] Then $\varphi_{i,z}$ is the ample model of $\fC_z$. On the other hand, \[K_{\fC_z}=f^*\KFA+R_z,\] where $R_z\geq0$ is $f$-exceptional and $\KFA$ is ample. Hence $f:\cY_z\to X$ is also the ample model of $\fC_z$. By uniqueness of the ample model, \[(\cV_{i,z},(\varphi_i)_*\cG_z,(\varphi_i)_*\cD_z,t)\simeq(X,\cF,B,t).\]

After stratifying $Z$, the families $\cV_i\to Z$, together with the pushforwards of $\cG$ and $\cE$, form finitely many log bounded families. By Noetherian induction, repeating the argument over the complement of the open subset treated above proves that $\cP$ is log bounded.
\end{enumerate}
\end{proof}

\subsection{Boundedness of polarized foliated surfaces}
In this subsection we prove the boundedness result for polarized adjoint foliated surfaces, and then apply it to derive boundedness of adjoint foliated surfaces of general type. 

The following lemma gives a uniform degree bound for foliations, which will be used to control the induced foliations in a bounded family.
\begin{lem}\label{bound foliation}
Let $\delta,v$ be two positive real numbers. Then there exists a positive real number $u$ depending only on $\delta,v$ satisfying the following: Assume that
\begin{enumerate}
    \item $\fA:=(X,\cF,B,t)$ is a klt adjoint foliated surface with $t\geq \delta$,
    \item $\KFA$ is pseudo-effective,
    \item $f:X\to Z$ is a birational morphism,
    \item $\fB:=(Z,\cG,B_Z,t)$ is the pushforward of $\fA$ on $Z$,
    \item $A$ is a very ample divisor on $Z$, and
    \item $\vol(\KFA+f^*A)\leq v$.
\end{enumerate}
Then $\KCG\cdot A\leq u$.

\begin{proof}
Let $g:X\to Y$ be the minimal model of $\KFA$ over $Z$, which is constructed by Corollary~\ref{relative adjoint MMP}. Denote the induced morphism $Y\to Z$ by $h$. Let $\fC$ be the pushforward of $\fA$ on $Y$. Then $\fC$ is klt, and $\KFC$ is nef over $Z$. By boundedness of the length of extremal rays, it follows that $\KFC+4h^*A$ is globally nef. Therefore,
\begin{align*}
    (\KFB+4A)\cdot A&=(\KFC+4h^*A)\cdot h^*A\\
    &\leq \vol(\KFC+5h^*A)\\
    &=\vol(\KFA+5f^*A)\leq 25v.
\end{align*}
Since \[A^2\leq \vol(\KFA+f^*A)\leq v,\] it follows that $(Z,A)$ belongs to a log bounded family. Hence, there exists a positive real number $s$ depending only on $v$ such that $K_Z+sA$ is pseudo-effective. Then \[K_Z\cdot A\geq -sA^2\geq -sv.\]
Since \[(t\KCG+(1-t)K_Z)\cdot A\leq 25v,\] we have \[\KCG\cdot A\leq \frac{1}{t}((1-t)sv+25v)\leq \frac{(s+25)v}{\delta}.\]
\end{proof}
\end{lem}

The following lemma gives a uniform positive lower bound for the log canonical thresholds of the polarizations. It will allow us to incorporate a fixed
small multiple of the polarization into the boundary.

\begin{lem}\label{lower bound of lct}
Let $\epsilon,\delta,v$ be three positive real numbers. Then there exists a positive real number $\tau$ depending only on $\epsilon,\delta,v$ satisfying the following: Assume that
\begin{enumerate}
\item $\fA:=(X,\cF,B,t)$ is an $\epsilon$-lc adjoint foliated surface with $t\leq 1-\delta$ and nef $\KFA$,
\item $L$ is a nef effective $\bR$-divisor on $X$,
\item $f:(X,\cF)\dashrightarrow (Y,\cG)$ is a birational map, where $\cG$ is the pushforward of $\cF$ on $Y$,
\item $\Sigma$ is a reduced divisor on $Y$ containing the support of the strict transform of $B$ and $L$, and the exceptional divisors of $f^{-1}$,
\item $A$ is a very ample divisor on $Y$ such that $A\cdot \Sigma\leq v$, and $A\cdot \KCG\leq v$,
\item $\pi:W\to X$, $\mu:W\to Y$ are common resolutions, and
\item $\vol(\pi^*L+\mu^*A)\leq v$.
\end{enumerate}
Then $(X,\cF,B+\tau(\nL+\frac{1}{1-t}\iL),t)$ is lc.

\begin{proof}
\begin{enumerate} [label=\textsl{Step} \arabic{enumi}., wide=13pt, itemsep=13pt]
\item In this step, we replace $Y$ by a bounded smooth model.

Since $A^2\leq \vol(\pi^*L+\mu^*A)\leq v$, $A\cdot \Sigma\leq v$, and $A\cdot \KCG\leq v$, Proposition~\ref{bdd criterion} implies that $(Y,\cG,\Sigma\cup\Supp(A))$ belongs to a log bounded family. Let $g:Y'\to Y$ be a bounded log resolution of $(Y,\Sigma\cup\Supp(A))$, $\cG'$ be the pullback of $\cG$, and $\Sigma'$ be the sum of the strict transform of $\Sigma$ and the reduced $g$-exceptional divisors, such that $(Y',\Sigma')$ is $\cG'$-adapted log smooth. By Lemma~\ref{delta-canonical family}, after replacing $(Y',\cG',\Sigma')$ by a log bounded higher model, we may furthermore assume that $\cG'$ is lc.

By construction, there exist a fixed positive integer $s$ and a very ample divisor $A'$ on $Y'$ such that $sg^*A-A'$ is pseudo-effective. Replacing $W$ by a higher model, we may assume that $\mu:W\to Y$ factors through $Y'$. If we denote $W\to Y'$ by $\nu$, then 
\[\vol(\pi^*L+\nu^*A')\leq \vol(\pi^*L+s\mu^*A)\leq s^2v.\]
From now on, we replace $(Y,\cG,\Sigma)$ by $(Y',\cG',\Sigma')$, $A$ by $A'$, and $v$ by $s^2v$.

\item In this step, we show the existence of the uniform lower bound of log canonical thresholds.

Denote $L_Y:=\mu_*\pi^*L$. Then we have
\[L_Y\cdot A=\pi^*L\cdot\mu^*A\leq \vol(\pi^*L+\mu^*A)\leq v.\]
Hence, the coefficient of each component of $\Supp(L_Y)$ is bounded from above by $v$. Now we take \[\tau:=\frac{\epsilon\delta}{2v}.\]
Since $t\leq 1-\delta$, by Lemma~\ref{adjoint terminal}, it follows that \[\fB:=(Y,\cG,(1-\epsilon)\Sigma+\tau(\nL_Y+\frac{1}{1-t}\iL_Y),t)\] is an lc adjoint foliated surface. 
Let \[\fA':=(X,\cF,B+\tau(\nL+\frac{1}{1-t}\iL),t).\] Then we conclude that \[K_{\fA'}=\KFA+\tau L\] is nef. Since $\fA=(X,\cF,B,t)$ is $\epsilon$-lc, and since $\Sigma$ contains the support of the strict transform of $B$ and exceptional divisors, we have \[\mu_*\pi^*K_{\fA'}\leq \KFB.\] By the negativity lemma, it follows that \[\pi^*K_{\fA'}\leq \mu^*\KFB.\] Therefore, $\fA'$ is also lc and we finish the proof.
\end{enumerate}
\end{proof}
\end{lem}

With these preparations in place, we are now ready to prove the main boundedness theorem for polarized adjoint foliated surfaces.
\begin{thm}\label{bdd of surfaces}
    Let $\epsilon,\delta,v$ be positive real numbers, and $I\subset\bQ\cap(0,1)$ be a finite set. Consider adjoint foliated surfaces $\fA:=(X,\cF,B,t)$ and integral divisors $N$ on $X$ such that
    \begin{enumerate}
        \item $\fA$ is $\epsilon$-lc,
        \item $t\in I$,
        \item the non-zero coefficients of $B$ are $\geq \delta$,
        \item either $\nB\in I$, or there exists an effective $\bQ$-divisor $\Delta\leq \nB$ such that $\Delta\in I$ and $\KCF+\Delta$ is pseudo-effective,
        \item $\KFA$ is nef,
        \item $N$ is nef and big, and
        \item $\vol(\KFA+N)\leq v$.
    \end{enumerate}
    Then the set of such adjoint foliated surfaces $(X,\cF, B,t)$ forms a log bounded family. If in addition $N\geq 0$, then the set of such adjoint foliated surfaces $(X,\cF,B+N,t)$ forms a log bounded family.

\begin{proof}
\begin{enumerate} [label=\textsl{Step} \arabic{enumi}., wide=13pt, itemsep=13pt]
\item In Steps 1 and 2, we first consider the special case where $\KFA+N$ is ample. In this step, we construct a birational model $(Y,\cH,\Sigma)$ of $(X,\cF,B)$ which belongs to a log bounded family.

Possibly decreasing $\delta$, we may assume that $\delta\leq t\leq 1-\delta$. By \cite[Lemma 5.3]{LX25NonalgebraictyNonabundant}, $K_\cF+\nB$ is pseudo-effective. Hence if $\nB\in I$, we take $\Delta=\nB$. By Proposition~\ref{effective birationality}, there exist two positive integers $m,n$ depending only on $\epsilon,I$ such that \[|m(t(\KCF+\Delta)+(1-t)K_X)+nN|\] defines a birational map. Let $\pi:W\to X$ be a foliated log resolution of $(X,\cF,B)$ such that
\[\pi^*(m(t(\KCF+\Delta)+(1-t)K_X)+nN)\sim M+F,\] where $F$ is the fixed part, and $M$ is the movable part, which is base point free and defines a birational morphism $\mu:W\to Y$. Denote $A:=\mu_*M$. Then $A$ is a very ample divisor on $Y$. Let \[L:=\pi_*M+\pi_*F+mt(\nB-\Delta)+m(1-t)B\sim m\KFA+nN.\] Then $L$ is an ample effective $\bR$-divisor on $X$, and $\vol(L)$ is bounded from above. Let $H\in |12M|$ be a general element in the linear system and $E=\on{Exc}(\pi)$. Define 
\[G:=\delta\red(\pi_*^{-1}B+F+E)+\frac{1}{2}H,\text{ and } \Sigma:=\Supp(G).\] 
Replacing $W$ by a higher model, we may assume that $(W,\Sigma)$ is log smooth. 

By \cite[Lemma 2.46]{Bir19AntiPluricanonicalSystemsFano}, $K_W+G$ is big. By \cite[Lemma 7.3]{HMX14ACCLCT}, there exists a positive real number $\lambda$ depending only on $\delta$ such that $K_W+\lambda G$ is big. Since the non-zero coefficients of $B$ are $\geq \delta$ and $\pi_*F$ is an integral divisor, we have 
\[\pi_*G=\delta\red(B+\pi_*F)+\frac{1}{2}\pi_*H\leq B+\pi_*F+\frac{1}{2}\pi_*H.\]
Choose $c>0$ such that \[1+c=\frac{1}{\delta}+\lambda c.\] Then we have 
\begin{align*}
\vol(K_W+\Sigma+10M)&\leq \vol(K_W+\frac{1}{\delta} G+10M+c(K_W+\lambda G))\\
&=\vol((1+c)(K_W+G)+10M)\\
&\leq \vol((1+c)(K_X+\pi_*G)+10\pi_*M)\\
&\leq \vol((1+c)(K_X+B+\pi_*F+6\pi_*M)+10\pi_*M)\\
&\leq \vol((1+c)(\frac{1}{m\delta}L+6L)+10L)
\end{align*}
is bounded from above. Let $\Sigma_Y:=\mu_*\Sigma$. By \cite[Lemma 3.2]{HMX13OnBirationalAutomorphisms},
\begin{align*}
    \Sigma_Y\cdot A=\Sigma\cdot M\leq \frac{2}{5}\vol(K_W+\Sigma+10M)
\end{align*}
is bounded from above. Define
\[\fB:=(W,\cG,\pi_*^{-1}B+(1-\epsilon)E,t),\] where $\cG$ is the pullback of $\cF$. Let $\cH$ be the pushforward of $\cG$ on $Y$. Since 
\[\vol(\KFB+\mu^*A)=\vol(\KFB+M)\leq \vol((m+1)\KFA+nN)\] is bounded from above, by Lemma~\ref{bound foliation}, it follows that $\KCH\cdot A$ is bounded from above. Therefore, applying Proposition~\ref{bdd criterion}, we conclude that $(Y,\cH,\Sigma_Y)$ belongs to a log bounded family. Replacing $\Sigma_Y$ by $\Supp(\Sigma_Y+A)$, we may assume that $\Sigma_Y$ contains $A$.

\item In this step, we prove the log boundedness of $(X,\cF,\Supp B)$ when $\KFA+N$ is ample.

Since \[\vol(\pi^*L+\mu^*A)\leq \vol(2L)\] is bounded from above, by Lemma~\ref{lower bound of lct}, there exists a fixed positive real number $\tau$ such that \[\fA':=(X,\cF,B+\tau(\nL+\frac{1}{1-t}\iL),t)\] is lc. Replacing $\tau$ by $\frac{\tau}{2}$, we may assume that $\fA'$ is $\frac{\epsilon}{2}$-lc. Since \[K_{\fA'}=\KFA+\tau L\] is ample, it follows from Theorem~\ref{Deriving boundedness from birational boundedness} that $\fA'$ belongs to a log bounded family. If moreover $N$ is an effective integral divisor, then by replacing $n$ with $n+1$ in Step 1, we may assume that $L\geq N$. Therefore, the log boundedness of $(X,\cF,\Supp(B+L))$ implies the log boundedness of $(X,\cF,\Supp(B+N))$.

\item In this step, we consider the general case where $\KFA+N$ is only nef and big.

By the above argument, \[\fA'=(X,\cF,B+\tau(\nL+\frac{1}{1-t}\iL),t)\] is $\frac{\epsilon}{2}$-lc, and $K_{\fA'}$ is nef and big. Let $h:X\to Z$ be the log canonical model of $\fA'$, which is constructed by Corollary~\ref{ample model}. Let $\fC$ be the pushforward of $\fA'$ on $Z$. Since \[K_{\fA'}=\KFA+\tau L=(m\tau+1)\KFA+n\tau N\]
and since $\KFA$ and $N$ are nef, it follows that $\KFA$ and $N$ are numerically trivial over $Z$. Therefore, by the negativity lemma, $\KFA$ and $N$ are $\bR$-linear trivial over $Z$. Then by Steps 1 and 2, $\fC$ belongs to a log bounded family. Since \[K_{\fA'}=h^*\KFC,\]
$h:X\to Z$ only extracts exceptional divisors with non-positive discrepancies with respect to $\fC$. Therefore, by Proposition~\ref{extract divisor in family}, $\fA'$ belongs to a log bounded family. In particular, $(X,\Supp(B+N))$ belongs to a log bounded family if $N$ is effective, as we can assume that $L\geq N$ as in Step 2. This completes the proof.
\end{enumerate}
\end{proof}
\end{thm}

As an application, we take the polarization to be a fixed multiple of the adjoint canonical divisor, and derive boundedness of adjoint foliated surfaces of general type with bounded volume.

\begin{cor}\label{bdd of surfaces of general type}
    Let $\epsilon,v$ be positive real numbers, and $I\subset\bQ\cap(0,1)$ be a finite set. Consider adjoint foliated surfaces $\fA:=(X,\cF,B,t)$ such that
    \begin{enumerate}
        \item $\fA$ is $\epsilon$-lc,
        \item $t\in I$, and the non-zero coefficients of $B$ are in $I$,
        \item $\KFA$ is nef and big, and
        \item $\vol(\KFA)\leq v$.
    \end{enumerate}
    Then the set of such adjoint foliated surfaces $(X,\cF,B,t)$ forms a log bounded family.
\end{cor}
\begin{proof}
    Since $t\in I$ and $B\in I$, there exists a positive integer $p$ depending only on $I$ such that $N:=p\KFA$ is integral. Moreover, we have \[\vol(\KFA+N)\leq (p+1)^2v.\]
    Therefore, the corollary follows from Theorem~\ref{bdd of surfaces}.
\end{proof}

\subsection{Applications to volumes}
We conclude this section with two applications to volumes. The first one gives a uniform positive lower bound of volumes of adjoint foliated surfaces of general type. 

\begin{cor}\label{lower bound of volume of general type}
Let $\epsilon$ be a positive real number, and $I\subset \bQ\cap(0,1)$ be a finite set. Then there exists a positive real number $v$ depending only on $\epsilon,I$ satisfying the following: 

If $\fA:=(X,\cF,B,t)$ is an $\epsilon$-lc adjoint foliated surface of general type such that $t\in I$ and the non-zero coefficients of $B$ are in $I$, then \[\vol(\KFA)\geq v.\]

\begin{proof}
This is an immediate consequence of Corollary~\ref{effective birationality of general type}.
\end{proof}
\end{cor}

The following result is in the same spirit as Jiao's discreteness theorem for volumes of integral divisors on Calabi--Yau type varieties \cite[Theorem~1.1]{Jiao25Discretenessvolume}. In our setting, the divisor whose volume is considered is the canonical divisor of an arbitrary rank one foliation on a Calabi--Yau adjoint foliated surface.

\begin{cor}\label{discreteness of volume of CY}
Let $\epsilon,\delta$ be two positive real numbers, and $I\subset \bQ\cap(0,1)$ be a finite set. Then there exists a discrete set $J\subset \bQ^{>0}$ depending only on $\epsilon,\delta,I$ satisfying the following. Assume that
\begin{enumerate}
\item $\fA:=(X,\cF,B,t)$ is an $\epsilon$-lc adjoint foliated surface with $t\in I$,
\item $\KFA\sim_{\bR}0$,
\item the non-zero coefficients of $B$ are $\geq \delta$,
\item either $\nB\in I$, or there exists an effective $\bQ$-divisor $\Delta\leq \nB$ such that $\Delta\in I$ and $\KCF+\Delta$ is pseudo-effective, and
\item $\cG$ is a rank one foliation on $X$.
\end{enumerate}
Then $\vol(\KCG)\in J\cup\{0\}$. In particular, if $\KCG$ is big, then $\vol(\KCG)$ is bounded from below by $\min J$.

\begin{proof}
We prove the stronger statement that the positive volumes of all integral divisors on $X$ form a discrete set depending only on $\epsilon, \delta, I$.

Let $D$ be an integral divisor on $X$. If $D$ is not big, then $\vol(D)=0$ and we are done. Now we assume that $D$ is big. It suffices to prove that for any positive real number $v$, if $\vol(D)\leq v$, then $\vol(D)$ is in a finite set.

Let $L\in |D|_\bQ$ be a general element. Take $\tau$ to be a sufficiently small positive real number such that \[\fB:=(X,\cF,B+\tau(\nL+\frac{1}{1-t}\iL),t)\] is klt. By Theorem~\ref{adjoint MMP}, we can run a $\KFB$-MMP, which is also a $D$-MMP. Let $f:X\to X'$ be the minimal model of $\KFB$. Let $\fA':=(X',\cF',B',t)$ be the pushforward of $\fA$, $D'$ be the pushforward of $D$ on $X'$. Then $D'$ is nef and big.

Since $\fA$ is Calabi--Yau, it follows that $\KFA=f^*K_{\fA'}$. Hence $\fA'$ is also an $\epsilon$-lc Calabi--Yau adjoint foliated surface. Then it follows from Proposition~\ref{control singularities} that $X'$ is $\epsilon$-lc. By Theorem~\ref{bdd of surfaces}, $X'$ belongs to a bounded family $\cP$ depending only on $\epsilon,\delta,v,I$. Therefore, by \cite[Theorem 1.2]{HJ26TotalCartierIndexBoundedFamily}, there exists a positive integer $N$ depending only on $\cP$, hence depending only on $\epsilon,\delta,v,I$, such that $ND'$ is Cartier. Thus $N^2\vol(D')\in \bZ^{>0}$, and hence \[\vol(D)=\vol(D')\] belongs to a finite set depending only on $\epsilon,\delta,v,I$.

We now take $D$ to be $\KCG$ and this completes the proof.
\end{proof}
\end{cor}

\bibliography{pfs}
\bibliographystyle{alphaurl}
\end{document}